\documentclass[12pt,oneside,english]{amsart}
\usepackage[T1]{fontenc}
\usepackage[latin9]{inputenc}
\usepackage[a4paper]{geometry}
\usepackage{babel}
\usepackage{prettyref}
\usepackage{mathtools}
\usepackage{amstext}
\usepackage{amsthm}
\usepackage{amssymb}
\usepackage{graphicx}
\usepackage{esint}
\usepackage[unicode=true,
 bookmarks=false,
 breaklinks=false,pdfborder={0 0 1},backref=section,colorlinks=false]
 {hyperref}

\makeatletter
\numberwithin{equation}{section}
\numberwithin{figure}{section}
\theoremstyle{plain}
\newtheorem{thm}{\protect\theoremname}
\theoremstyle{plain}
\newtheorem{cor}[thm]{\protect\corollaryname}
\theoremstyle{plain}
\newtheorem{prop}[thm]{\protect\propositionname}
\theoremstyle{plain}
\newtheorem{lem}[thm]{\protect\lemmaname}

\usepackage{babel}
\makeatletter
\@namedef{subjclassname@2020}{%
  \textup{2020} Mathematics Subject Classification}
\makeatother

\makeatother

\providecommand{\corollaryname}{Corollary}
\providecommand{\lemmaname}{Lemma}
\providecommand{\propositionname}{Proposition}
\providecommand{\theoremname}{Theorem}

\begin{document}
\title[Quantitative Weyl law]{ Quantitative version of Weyl's law}
\author{Nikhil Savale}
\thanks{N. S. is partially supported by the DFG funded project CRC/TRR 191.}
\address{Universität zu Köln, Mathematisches Institut, Weyertal 86-90, 50931
Köln, Germany}
\email{nsavale@math.uni-koeln.de}
\begin{abstract}
We prove a general estimate for the Weyl remainder of an elliptic,
semiclassical pseudodifferential operator in terms of volumes of recurrence
sets for the Hamilton flow of its principal symbol. This quantifies
earlier results of Volovoy \cite{Volovoy-1990-II,Volovoy-1990-I}.
Our result particularly improves Weyl remainder exponents for compact
Lie groups and surfaces of revolution. And gives a quantitative estimate
for Bérard's Weyl remainder in terms of the maximal expansion rate
and topological entropy of the geodesic flow.
\end{abstract}

\maketitle

\section{Introduction}

Let $X^{n}$ be a smooth, compact, $n$-dimensional manifold. Let
$P\in\Psi_{h,\textrm{cl}}^{m}\left(X\right)$ a self-adjoint, semiclassical
($h$-)pseudodifferential operator. We assume that the total symbol
of $P$ has a classical expansion 
\begin{align}
P & =p_{h}^{W},\label{eq:classical pseudodifferential operator}\\
p_{h} & \sim p_{0}+hp_{1}+\ldots,
\end{align}
for $p_{j}\in C^{\infty}\left(T^{*}X\right)$, $j=0,1,2,\ldots$.
Let $\left[a,b\right]\subset\mathbb{R}$ and assume that the principal
symbol $p_{0}$ is elliptic in this interval, i.e. $p_{0}^{-1}\left[a,b\right]\Subset T^{*}X$
is compact. The spectrum of $P_{h}$ in $\left[a,b\right]$ is then
discrete for $h$ sufficiently small. Denote by $N_{h}\left[a,b\right]$
the number of eigenvalues contained in this interval. Further suppose
that $a,b$ are non-critical for the principal symbol $p_{0}$, whereby
the the energy levels $\Sigma_{a}\coloneqq p_{0}^{-1}\left(a\right)\subset T^{*}X,\,\Sigma_{b}\coloneqq p_{0}^{-1}\left(b\right)\subset T^{*}X$
are well-defined hypersurfaces and so is the Hamiltonian flow $e^{tH_{p_{0}}}$
for the principal symbol on them. We denote by the shorthand $e^{tH_{p_{0}}^{a}}\coloneqq\left.e^{tH_{p_{0}}}\right|_{\Sigma_{a}}$
the Hamilton flow restricted to the given energy level. And by $T_{0}^{a}>0$
the shortest period of the restricted Hamilton flow $e^{tH_{p_{0}}^{a}}$.
These energy levels also carry the canonical Liouville volume form
$d\nu$ such that $dp_{0}\wedge d\nu$ is the canonical symplectic
volume form in a neighbourhood of these levels in $T^{*}X$. 

Next, define the recurrence sets 
\begin{equation}
S_{T,\varepsilon}^{a}\coloneqq\left\{ \left(x,\xi\right)\subset\Sigma_{a}|\exists t\in\left[\frac{1}{2}T_{0}^{a},T\right]\textrm{ s.t. }d\left(e^{tH_{p_{0}}}\left(x,\xi\right),\left(x,\xi\right)\right)\leq\varepsilon\right\} \label{eq:recurrence set}
\end{equation}
and $S_{T,\varepsilon}^{a,e}\coloneqq\left\{ \left(x,\xi\right)\subset\Sigma_{a}|d\left(\left(x,\xi\right),S_{T,\varepsilon}^{a}\right)\leq\varepsilon\right\} $
for each $T>\frac{1}{2}T_{0}^{a}$, $\varepsilon>0$. These are defined
with respect to a distance function $d$ on phase space that is equivalent
to a manifold distance, although the dependence of the above \prettyref{eq:recurrence set}
on $d$ is supressed in the notation. 

Next we further, for each $\ell>0$, define the Ehrenfest time via
\begin{align}
T_{E}^{\ell,a}\left(h\right) & \coloneqq\begin{cases}
\infty; & \textrm{if }\exists C>0\textrm{ s.t.}\;\left|\partial^{\alpha}e^{tH_{p_{0}}^{a}}\right|\lesssim\left|t\right|^{C},\;\forall\alpha\in\mathbb{N}_{0}^{2n-1},\\
\frac{\left|\ln h\right|}{\Lambda_{\textrm{max}}^{a}+\ell}; & \textrm{otherwise},
\end{cases}\label{eq:Ehrenfest time}\\
\textrm{ where}\quad\Lambda_{\textrm{max}}^{a} & \coloneqq\limsup_{t\rightarrow\infty}t^{-1}\sup_{x\in\Sigma_{a}}\ln\left|de^{tH_{p_{0}}}\left(x\right)\right|\label{eq:max expansion rate}
\end{align}
are the sup norm of the Jacobian and the maximum expansion rate of
the Hamilton flow on the energy level $\Sigma_{a}$ respectively.

Our main result is the following general estimate on the Weyl remainder
for $P$.
\begin{thm}
\label{thm:general Weyl law} Let $P\in\Psi_{h,\textrm{cl}}^{m}\left(X\right)$
be a self-adjoint, $h$-pseudodifferential operator that is elliptic
in the interval $\left[a,b\right]$ and whose principal symbol is
non-critical at its endpoints. For each $\varepsilon=ch^{\delta}$,
$\delta\in\left[0,\frac{1}{2}\right)$, $c,\ell>0$ and $T\leq\left(\frac{1}{2}-\delta\right).\max\left\{ T_{E}^{\ell,a}\left(h\right),T_{E}^{\ell,b}\left(h\right)\right\} $,
the Weyl counting function of the interval satisfies 
\begin{align}
N_{h}\left[a,b\right] & =\left(2\pi h\right)^{-n}\left[\textrm{vol }p_{0}^{-1}\left[a,b\right]+h\left(\int_{\Sigma_{a}}p_{1}d\nu-\int_{\Sigma_{b}}p_{1}d\nu\right)+hR_{h}\right]\label{eq:Weyl leading term}\\
\textrm{with }\quad\left|R_{h}\right| & \leq\left(\nu\left(\Sigma_{a}\right)+\nu\left(\Sigma_{b}\right)\right)T^{-1}+O\left(\nu\left(S_{T,\varepsilon}^{a,e}\right)+\nu\left(S_{T,\varepsilon}^{b,e}\right)+T^{-2}+h^{1-2\delta}\right)\label{eq:general remainder estimate}
\end{align}
as $h\rightarrow0$.
\end{thm}

Interesting specializations of the above arise depending on estimates
for volumes of recurrence sets. Firstly, it recovers the classical
results of Hörmander \cite{Hormander68} and Duistermaat-Guillemin
\cite{Duistermaat-Guillemin}, in their semiclassical form cf. \cite{Ivrii-newbook-I-2019,Petkov-Robert-85}.
Namely, under the non-criticality assumption on the endpoints, one
first obtains $N_{h}\left[a,b\right]=\left(2\pi h\right)^{-n}\left[\textrm{vol}p_{0}^{-1}\left[a,b\right]+O\left(h\right)\right]$
from \prettyref{eq:Weyl leading term} which is the semiclassical
version of Hörmander's Weyl law. Secondly, and further assuming that
the set of periodic Hamilton trajectories is of measure zero on the
two energy levels $\Sigma_{a},\Sigma_{b}$, one has the $R_{h}=o\left(1\right)$
in \prettyref{eq:Weyl leading term}. This is the semiclassical version
of Duistermaat and Guillemin's Weyl law. This follows on letting $\varepsilon,T$
be $h-$independent. Since the recurrence set volume approaches the
measure of the set of periodic Hamilton trajectories as $\varepsilon\rightarrow0$
and $T\rightarrow\infty$, one recovers Duistermaat-Guillemin.

Another interesting specialization of the general estimate is when
the Hamilton flow is a contact Anosov flow. In this case the recurrence
set volumes can be shown to satisfy exponential estimates in time
$\nu\left(S_{T,\varepsilon}^{a,e}\right)=O\left(\varepsilon^{2n-1}e^{\left(2n-1\right)\Lambda_{\textrm{max}}^{a}T}\right)$
with the exponent again being the maximal expansion rate \prettyref{eq:max expansion rate}
(see Section \prettyref{subsec:Anosov-flows} below). By an appropriate
choice of $\varepsilon,T$ in this case, one obtains a logarithmic
improvement in the Weyl law. A particular example of contact Anosov
Hamiltonian flows are geodesic flows on negatively curved manifolds.
Thus, if we particularly specialize to the case when $P_{h}=h^{2}\Delta_{g}$
is the semiclassical Laplacian for a negatively curved Riemannian
metric $\left(X,g\right)$, with $a<0$, $b=1$, we obtain the following
quantitative version of Bérard's Weyl law.
\begin{cor}
\label{cor:Quant. Ber.} Let $\left(X^{n},g\right)$ be a compact,
negatively curved Riemannian manifold. The Weyl counting function
for the semiclassical Laplacian $P_{h}\coloneqq h^{2}\Delta_{g}$
satisfies the asymptotics 
\begin{align}
N_{h}\left[0,1\right] & =\left(2\pi h\right)^{-n}\textrm{vol}\left(S^{*}X\right)\left[1+hR_{h}\right]\label{eq:Berard leading}\\
\textrm{where }\quad\left|R_{h}\right| & \leq4\Lambda_{\textrm{max}}\left|\ln h\right|^{-1}+o\left(\left|\ln h\right|^{-1}\right)\label{eq:lambda max in remainder}\\
 & \leq\frac{16}{2n-1}\mathtt{h}_{\textrm{top}}\left|\ln h\right|^{-1}+o\left(\left|\ln h\right|^{-1}\right),\label{eq:entropy in remainder}
\end{align}
as $h\rightarrow0$, in terms maximal expansion rate $\Lambda_{\textrm{max}}$
and the topological entropy $\mathtt{h}_{\textrm{top}}$ of the geodesic
flow on the cosphere bundle $S^{*}X$.
\end{cor}

The next specialization concerns the semiclassical Laplacian $P_{h}=h^{2}\Delta_{g}$
corresponding to a bi-invariant metric on compact a Lie group $G$.
The corresponding Hamiltonian flow is again the geodesic flow which
is given simply by the group action. The Ehrenfest time \prettyref{eq:Ehrenfest time}
is infinite in this case. While the recurrence set volume for the
geodesic flow on the unit cosphere bundle can be shown to satisfy
the bound $\nu\left(S_{T,\varepsilon}^{1,e}\right)=O\left(\varepsilon^{p-1}T^{p}\right),\text{ with }p=\textrm{rk}G,$
being the rank of the Lie group (see Section \prettyref{subsec:Compact-Lie-Groups}).
This gives the Weyl law below as a corollary. 
\begin{cor}
\label{cor:Lie group corollary} Let $G$ be a compact Lie group equipped
with a bi-invariant Riemannian metric. The Weyl counting function
for the semiclassical Laplacian $P_{h}\coloneqq h^{2}\Delta_{g}$
satisfies the asymptotics 
\begin{equation}
N_{h}\left[0,1\right]=\left(2\pi h\right)^{-n}\textrm{vol}\left(S^{*}X\right)\left[1+O\left(h^{1+\frac{p-1}{3p+1}}\right)\right],\label{eq:Lie group Weyl law}
\end{equation}
as $h\rightarrow0$, where $p=\textrm{rk }G$ is the rank of the Lie
group. 
\end{cor}

Our final specialization concerns the semiclassical Laplacian $P_{h}=h^{2}\Delta_{g}$
on surfaces of revolution $X$. The corresponding Hamiltonian flow
is again the geodesic flow. In this case, it has a simple description
as an ordinary differential equation on the base surface (see \prettyref{subsec:Surface-of-revolution}
below). The Ehrenfest time \prettyref{eq:Ehrenfest time} is again
infinite. Assuming the surface to be strictly convex, it has an equator;
the unique rotationally invariant geodesic $\gamma_{E}\subset X$
of maximal length. For any point $x\in\gamma_{E}$ the equatorial
return map $\theta:S_{x}^{*}X\rightarrow\gamma_{E}$ is well defined.
This is the map sending a unit covector $\xi\in S_{x}^{*}X$ to the
first return point $\theta\left(\xi\right)$ of the geodesic $\exp_{x}\left(t\xi\right)$
to the equator $\gamma_{E}$. The equatorial return map has an order
of vanishing defined as 
\begin{align*}
r_{\xi_{0}} & \coloneqq\textrm{ord}_{\xi_{0}}\left[\theta\left(\xi\right)-\theta\left(\xi_{0}\right)\right]=\min\left\{ l|\partial_{\xi}^{l}\left[\theta\left(\xi\right)-\theta\left(\xi_{0}\right)\right]\left(\xi_{0}\right)\neq0\right\} ,\quad r_{\xi_{0}}\geq1,
\end{align*}
for each covector $\xi_{0}\in S_{x}^{*}X$. The function $r_{\xi_{0}}$
is upper semi-continuous in $\xi_{0}$ and it is clearly independent
of $x$ by rotational invariance. We set 
\begin{align}
r & \coloneqq\sup_{\xi_{0}\in S_{x}^{*}X}r_{\xi_{0}}\nonumber \\
 & =\sup_{\xi_{0}\in S_{x}^{*}X}\textrm{ord}_{\xi}\left[\theta\left(\xi\right)-\theta\left(\xi_{0}\right)\right]\label{eq:order of vanishing of E. return map}
\end{align}
to be its supremum over $\xi_{0}\in S_{x}^{*}X$. There a standard
formulas to calculate $\theta$ and $r$ in terms of a given equation
for the surface (see \cite{Besse78-book}). For a real analytic surface
which is not Zoll we have $r<\infty$ is finite, while for a generic
surface one has $r=1$.

Under some further assumptions stated below, the recurrence set volume
for the geodesic flow on the unit cosphere bundle can be shown to
satisfy $\nu\left(S_{T,\varepsilon}^{1,e}\right)=O\left(\varepsilon^{\frac{1}{r}}T^{1-\frac{1}{r}}\right)$
for $\varepsilon,T$ sufficiently small. This gives the Weyl law below
as a corollary. 
\begin{cor}
\label{cor: Surface of revolution corollary} Let $\left(X^{2},g\right)$
be a compact, strictly convex surface of revolution. The Weyl counting
function for the semiclassical Laplacian $P_{h}\coloneqq h^{2}\Delta_{g}$
satisfies the asymptotics 
\begin{equation}
N_{h}\left[0,1\right]=\left(2\pi h\right)^{-2}\textrm{vol}\left(S^{*}X\right)\left[1+O\left(h^{1+\frac{1}{4r-1}}\right)\right],\label{eq: surface revolution Weyl law}
\end{equation}
as $h\rightarrow0$. Here $r$ \prettyref{eq:order of vanishing of E. return map}
is the maximum order of vanishing of the equatorial return map.
\end{cor}

The corollaries are all based on a judicious choice of $\varepsilon$,
$T$ in the main \prettyref{thm:general Weyl law}. 

The general estimate in \prettyref{thm:general Weyl law}, or the
arguments for it, are seemingly folklore. Less precise versions of
it are explained in the articles of Volovoy \cite{Volovoy-1990-II,Volovoy-1990-I}
as well as the recent book \cite[Sec. 4.5.4]{Ivrii-newbook-I-2019}
of Ivrii. However the identification of the various exponents and
constants in our corollaries does not appear in them. Our own motivation
came from the recent article of the author \cite{Savale2020-hyperbolicity}
where an analogous result is proved for the coupled Dirac operator.
Our method is based on a modification of the argument in the book
of Dimassi-Sjöstrand \cite[Ch. 11]{Dimassi-Sjostrand}.

The first Corollary \prettyref{cor:Quant. Ber.} quantifies Bérard's
Weyl law \cite{Berard77} for negatively curved manifolds. Its identification
of the maximal expansion rate \prettyref{eq:lambda max in remainder}
or topological entropy \prettyref{eq:entropy in remainder} appears
to be new. It might however be of little interest since there are
stronger conjectures. For instance it is conjectured that for a non-arithmetic
or generic surface the remainder is $R_{h}=O\left(h^{2-}\right)$.
Although no better improvement is known since Bérard's article.

The second and third corollaries were also considered by Volovoy.
He proved the remainder estimates $O\left(h^{1+\delta}\right)$, for
some $\delta>0$, in \prettyref{eq:Lie group Weyl law}, \prettyref{eq: surface revolution Weyl law}.
And our corollaries quantify the value of $\delta$ that can be chosen.
In Corollary \prettyref{cor:Lie group corollary}, the case of the
torus $G=\mathbb{T}^{n}$ is particularly studied (see \cite[Ch. 3]{Fricker-book-1982}).
In this case, the sharp Weyl remainder is conjecturally
\[
R_{h}=\begin{cases}
O\left(h^{\frac{3}{2}-}\right), & \textrm{for }\;n=2,\;\textrm{ (Gauss circle problem)}\\
O\left(h^{2-}\right), & \textrm{for }\;n=3,4,\\
O\left(h^{2}\right), & \textrm{for }\;n\geq5.
\end{cases}
\]
\footnote{Here $O\left(h^{\alpha-}\right)$ denotes a term that is $O\left(h^{\alpha-\epsilon}\right)$
for each $\epsilon>0$.} This is known in dimensions $n\geq4$. While $R_{h}=O\left(h^{\frac{285}{208}-}\right),O\left(h^{\frac{5}{3}-}\right)$
are the best known results in low dimensions $n=2,3$ respectively
\cite{Huxley2002}. Our Corollary \prettyref{cor:Lie group corollary}
is worse than these known results for the torus. However, the estimate
\prettyref{eq:Lie group Weyl law} for an arbitrary Lie group appears
to be new. The problem of determining the sharp exponent in the Weyl
is further unknown for general $G$. 

Finally, we remark that although it is not our interest here, our
method almost certainly gives similar improvements for local Weyl
laws as well as $L^{2}$, $L^{\infty}$ norms and averages of eigenfunctions
etc. 

The article is organized as follows. In \prettyref{sec:Preliminaries}
we begin with some preliminaries on semiclassical analysis and dynamical
systems. In Section \prettyref{subsec:General Weyl law} we prove
the general Weyl's law in our main \prettyref{thm:general Weyl law}
based on a modification of the argument in \cite{Dimassi-Sjostrand}.
In \prettyref{sec:Examples-of-recurrence} we consider recurrence
sets for several dynamical systems including Anosov flows \prettyref{subsec:Anosov-flows}
as well as geodesic flows on Lie groups \prettyref{subsec:Compact-Lie-Groups}
and surfaces of revolution \prettyref{subsec:Surface-of-revolution}.
The final \prettyref{sec:Proofs-of-the-corollaries} proves the three
Corollaries \prettyref{cor:Quant. Ber.}, \prettyref{cor:Lie group corollary}
and \prettyref{cor: Surface of revolution corollary} based on the
volume bounds from \prettyref{sec:Examples-of-recurrence}. 

\section{\label{sec:Preliminaries} Preliminaries}

In this section we review some preliminary notions from semiclassical
analysis and dynamical systems that are used in the article.

\subsection{Semiclassical analysis}

First we begin with some requisite facts from semi-classical analysis
that shall be used in the paper, \cite{Dimassi-Sjostrand,GuilleminSternberg-Semiclassical}
provide the standard references. The symbol space $S^{m}\left(\mathbb{R}^{2n}\right)$
is defined as the space of maps $a:\left(0,1\right]_{h}\rightarrow C^{\infty}\left(\mathbb{R}_{x,\xi}^{2n}\mathbb{C}\right)$
for which each semi-norm
\[
\left\Vert a\right\Vert _{\alpha,\beta}:=\text{sup}_{\substack{x,\xi}
,h}\langle\xi\rangle^{-m+|\beta|}\left|\partial_{x}^{\alpha}\partial_{\xi}^{\beta}a(x,\xi;h)\right|,\quad\alpha,\beta\in\mathbb{N}_{0}^{n},
\]
is finite. The more refined class $a\in S_{{\rm cl\,}}^{m}\left(\mathbb{R}^{2n};\mathbb{C}^{l}\right)$
of classical symbols consists of those for which there exists an $h$-independent
sequence $a_{k}$, $k=0,1,\ldots$, of symbols satisfying
\begin{equation}
a-\left(\sum_{k=0}^{N}h^{k}a_{k}\right)\in h^{N+1}S^{m}\left(\mathbb{R}^{2n}\right),\;\forall N.\label{eq:classical symbolic expansion}
\end{equation}
 Any given $a\in S^{m}\left(\mathbb{R}^{2n}\right),S_{{\rm cl\,}}^{m}\left(\mathbb{R}^{2n}\right)$
in one of the symbol classes above defines a one-parameter family
of operators $a^{W}\in\Psi^{m}\left(\mathbb{R}^{2n}\right),\Psi_{{\rm cl\,}}^{m}\left(\mathbb{R}^{2n}\right)$
via Weyl quantization. The Schwartz kernel of the quantization is
given by 
\[
a^{W}\coloneqq\frac{1}{\left(2\pi h\right)^{n}}\int e^{i\left(x-y\right).\xi/h}a\left(\frac{x+y}{2},\xi;h\right)d\xi.
\]
The above pseudodifferential classes of operators are closed under
the usual operations of composition and formal-adjoint. Furthermore
the classes are invariant under changes of coordinates. Thus one may
invariantly define the classes of operators $\Psi^{m}\left(X\right),\Psi_{{\rm cl\,}}^{m}\left(X\right)$
acting on $C^{\infty}\left(X;\mathbb{C}\right)$ on a smooth compact
manifold $X$. 

The principal symbol of a classical pseudodifferential operator $A\in\Psi_{{\rm cl\,}}^{m}\left(X;E\right)$
is defined as an element in $\sigma\left(A\right)\in S^{m}\left(X\right)\subset C^{\infty}\left(T^{*}X;\mathbb{C}\right).$
It is given by $\sigma\left(A\right)=a_{0}$ the leading term in the
symbolic expansion \prettyref{eq:classical symbolic expansion} of
its full Weyl symbol. The principal symbol is multiplicative, commutes
with adjoints and fits into a symbol exact sequence 
\begin{align}
\sigma\left(AB\right) & =\sigma\left(A\right)\sigma\left(B\right)\nonumber \\
\sigma\left(A^{*}\right) & =\overline{\sigma\left(A\right)}\nonumber \\
0\rightarrow h\Psi_{{\rm cl\,}}^{m}\left(X\right)\rightarrow & \Psi_{{\rm cl\,}}^{m}\left(X\right)\xrightarrow{\sigma}S^{m}\left(X\right),\label{eq:properties of symbol}
\end{align}
with the formal adjoint defined with respect to an auxiliary density.
The quantization map 
\begin{align}
\textrm{Op}:S^{m}\left(X\right) & \rightarrow\Psi_{{\rm cl\,}}^{m}\left(X\right)\quad\textrm{ satisfying }\nonumber \\
\sigma\left(\textrm{Op}\left(a\right)\right) & =a\in S^{m}\left(X\right)\label{eq:quantization map}
\end{align}
gives an inverse to the principal symbol map. We sometimes use the
alternate notation $\textrm{Op}\left(a\right)=a^{W}$. The quantization
map above is however non-canonical and depends on the choice of a
coordinate atlas as well as a subordinate partition of unity. From
the multiplicative property of the symbol \prettyref{eq:properties of symbol},
it then follows that $\left[a^{W},b^{W}\right]\in h\Psi_{{\rm cl\,}}^{m-1}\left(X\right)$.
Its principal symbol is given by the Poisson bracket 
\[
\frac{i}{h}\sigma\left(\left[a^{W},b^{W}\right]\right)=\left\{ a,b\right\} \in S^{m}\left(X\right).
\]

Each $A\in\Psi_{{\rm cl\,}}^{m}\left(X\right)$ has a wavefront set
defined invariantly as a subset $WF\left(A\right)\subset\overline{T^{*}X}$
of the fibrewise radial compactification of the cotangent bundle $T^{*}X$.
It is locally defined as follows, $\left(x_{0},\xi_{0}\right)\notin WF\left(A\right)$,
$A=a^{W}$, if and only if there exists an open neighborhood $\left(x_{0},\xi_{0};0\right)\in U\subset\overline{T^{*}X}\times\left(0,1\right]_{h}$
such that $a\in h^{\infty}\left\langle \xi\right\rangle ^{-\infty}C^{k}\left(U;\mathbb{C}^{l}\right)$
for all $k$. The wavefront set satisfies the following basic properties
under addition, multiplication and adjoints 
\begin{align*}
WF\left(A+B\right) & \subset WF\left(A\right)\cup WF\left(B\right),\\
WF\left(AB\right) & \subset WF\left(A\right)\cap WF\left(B\right)\quad\textrm{and }\\
WF\left(A^{*}\right) & =WF\left(A\right).
\end{align*}
The wavefront set $WF\left(A\right)=\emptyset$ is empty if and only
if $A\in h^{\infty}\Psi^{-\infty}\left(X;E\right)$ while we say that
two operators $A=B$ microlocally on $U\subset\overline{T^{*}X}$
if $WF\left(A-B\right)\cap U=\emptyset$. 

An operator $A\in\Psi_{{\rm cl\,}}^{m}\left(X\right)$ is said to
be elliptic if $\left\langle \xi\right\rangle ^{m}\sigma\left(A\right)^{-1}$
exists and is uniformly bounded on $T^{*}X$. If $A\in\Psi_{{\rm cl\,}}^{m}\left(X\right)$,
$m>0$, is formally self-adjoint, such that $A+i$ is elliptic, then
it is essentially self-adjoint (with domain $C_{c}^{\infty}\left(X;E\right)$)
as an unbounded operator on $L^{2}\left(X;E\right)$. Beals's lemma
further implies that its resolvent $\left(A-z\right)^{-1}\in\Psi_{{\rm cl\,}}^{-m}\left(X\right)$,
$z\in\mathbb{C}$, $\textrm{Im}z\neq0$, exists and is pseudo-differential.
The Helffer-Sjöstrand formula now expresses the function $f\left(A\right)$,
$f\in\mathcal{S}\left(\mathbb{R}\right)$, of such an operator in
terms of its resolvent 
\[
f\left(A\right)=\frac{1}{\pi}\int_{\mathbb{C}}\bar{\partial}\tilde{f}\left(z\right)\left(A-z\right)^{-1}dzd\bar{z},
\]
with $\tilde{f}$ denoting an almost analytic continuation of $f$.
One further has 
\begin{align}
WF\left(f\left(A\right)\right) & \subset\Sigma_{\textrm{spt}\left(f\right)}^{A}\label{eq:wavefront function of operator}\\
 & \coloneqq\bigcup_{\lambda\in\textrm{spt}\left(f\right)}\Sigma_{\lambda}^{A}\\
\textrm{where }\quad\Sigma_{\lambda}^{A} & =\left\{ \left(x,\xi\right)\in T^{*}X|\sigma\left(A\right)\left(x,\xi\right)=\lambda\right\} \label{eq:energy level}
\end{align}
is classical $\lambda$-energy level of $A$. 

\subsubsection{\label{subsec:The-class-,}The class $\Psi_{\delta}^{m}\left(X\right)$}

We shall also need a more exotic class of scalar symbols $S_{\delta}^{m}\left(\mathbb{R}^{2n};\mathbb{C}\right)$
defined for each $0\leq\delta<\frac{1}{2}$. A function $a:\left(0,1\right]_{h}\rightarrow C^{\infty}\left(\mathbb{R}_{x,\xi}^{2n};\mathbb{C}\right)$
is said to be in this class if and only if 
\begin{equation}
\left\Vert a\right\Vert _{\alpha,\beta}:=\text{sup}_{\substack{x,\xi}
,h}h^{\left(\left|\alpha\right|+\left|\beta\right|\right)\delta}\left|\partial_{x}^{\alpha}\partial_{\xi}^{\beta}a(x,\xi;h)\right|\label{eq: delta pseudodifferential estimates}
\end{equation}
is finite $\forall\alpha,\beta\in\mathbb{N}_{0}^{n}.$ This class
of operators is also closed under the standard operations of composition,
adjoint and changes of coordinates; allowing for the definition of
the same exotic pseudo-differential algebra $\Psi_{\delta}^{m}\left(X\right)$
on a compact manifold. The class $S_{\delta}^{m}\left(X\right)$ is
a family of functions $a:\left(0,1\right]_{h}\rightarrow C^{\infty}\left(T^{*}X;\mathbb{C}\right)$
satisfying the estimates \prettyref{eq: delta pseudodifferential estimates}
in every coordinate chart and induced trivialization. Such a family
can be quantized to $a^{W}\in\Psi_{\delta}^{m}\left(X\right)$ satisfying
$a^{W}b^{W}=\left(ab\right)^{W}+h^{1-2\delta}\Psi_{\delta}^{m+m'-1}\left(X\right)$
for another $b\in S_{\delta}^{m'}\left(X\right)$. The operators in
$\Psi_{\delta}^{0}\left(X\right)$ are uniformly bounded on $L^{2}\left(X\right)$.
Finally, the wavefront an operator $A\in\Psi_{\delta}^{m}\left(X;E\right)$
is similarly defined and satisfies the same basic properties as before.

\subsection{\label{subsec:Dynamical-invariants} Dynamical invariants}

Here we state some facts on dynamical systems and their invariants
that are used in the article. Their proofs can be found in the texts
\cite{Ledrappier82-notes,Mane87-book}.

Let $V\in C^{\infty}\left(TY\right)$ be a smooth vector field on
a compact Riemannian manifold $\left(Y,g^{TY}\right)$ of dimension
$m$. A point $y\in Y$ is said to be a \textit{regular point} for
the flow of $V$ if there is a sequence of numbers $\lambda_{1}\left(y\right)>\ldots>\lambda_{k}\left(y\right)$
and decomposition for the tangent space $T_{y}Y=\oplus_{j=1}^{k}E_{j}\left(y\right)$
such that 
\begin{equation}
\limsup_{t\rightarrow\infty}t^{-1}\ln\left|de^{tV}\left(y\right)u\right|=\lambda_{j}\left(y\right),\label{eq:Lyapunoc exponents}
\end{equation}
$\forall u\in E_{j}\left(y\right),1\leq j\leq k$. The set of regular
points $R\subset Y$ is a Borel set of full measure. The numbers $\lambda_{j}\left(y\right)$
and subspaces $E_{j}\left(y\right)$ are referred to as the \textit{Lyapunov
exponents} and \textit{eigenspaces} of the flow at $y\in Y$ respectively.
The maximal expansion rate of the flow is the supremum of the maximum
of these 
\begin{align}
\Lambda_{\textrm{max}}\left(V\right) & \coloneqq\limsup_{t\rightarrow\infty}t^{-1}\sup_{y\in Y}\ln\left|de^{tV}\left(y\right)\right|\nonumber \\
 & =\sup_{y\in R\subset Y}\max_{j=1,\ldots,k}\lambda_{j}\left(y\right).\label{eq:lambda max as max antropy}
\end{align}
It is well known that $\Lambda_{\textrm{max}}\left(V\right)$ is an
upper semi-continuous function of $V$.

Additionally, it is useful to define the sum of the positive Lyapunov
exponents as the Borel function 
\begin{align}
\chi & :R\rightarrow\mathbb{R}\nonumber \\
\chi\left(y\right) & \coloneqq\sum_{\lambda_{j}\left(y\right)>0}\lambda_{j}\left(y\right)\textrm{dim }E_{j}\left(y\right).\label{eq:sum of positive lyapunov}
\end{align}
In case there is no positive Lyapunov exponent, we set $\chi\left(y\right)\coloneqq0$.

Next, let $\mu_{Y}$ be a Borel probability measure on $Y$ invariant
by the flow of $V$. The measure theoretic entropy of a partition
$\mathcal{P}=\left\{ P_{1},\ldots,P_{N}\right\} $ of $Y$ into measurable
subsets is defined to be $H\left(\mathcal{P}\right)\coloneqq-\sum_{P\in\mathcal{P}}\mu_{Y}\left(P\right)\ln\mu_{Y}\left(P\right)$.
The measure theoretic entropy of the flow with respect to the partition
is set to be 
\[
H\left(\mathcal{P},V\right)\coloneqq\lim_{N\rightarrow\infty}N^{-1}H\left(\bigvee_{j=0}^{N}e^{-jV}\mathcal{P}\right).
\]
Here $\mathcal{P}_{1}\bigvee\mathcal{P}_{2}$ above denotes the minimal
common refinement of two partitions $\mathcal{P}_{1},\mathcal{P}_{2}$.
The supremum of the above over all finite measurable partitions is
the measure theoretic entropy of the flow
\begin{equation}
h_{\mu_{Y}}\left(V\right)=\sup_{\mathcal{P}}H\left(\mathcal{P},V\right).\label{eq:measure theoretic entropy}
\end{equation}
While the topological entropy is the supremum of the above over all
the set of all $V$-invariant Borel probability measures $\mathcal{M}_{V}$
\begin{equation}
h_{\textrm{top}}\left(V\right)\coloneqq\sup_{\mu_{Y}\in\mathcal{M}_{V}}h_{\mu_{Y}}\left(V\right).\label{eq:topological entropy}
\end{equation}

The relation between the measure theoretic entropy and the Lyapunov
exponents is given by the Margulis-Ruelle inequality \cite[Thm. 10.2]{Mane87-book}
\begin{equation}
h_{\mu_{Y}}\left(V\right)\leq\int_{Y}\chi\left(y\right)d\mu_{Y}.\label{eq:Ruelle Margulis inequality}
\end{equation}
In case $\mu_{Y}$ is absolutely continuous with respect to a smooth
measure then Pesin's formula says that one has equality in the above.
Following the definitions \prettyref{eq:lambda max as max antropy},
\prettyref{eq:sum of positive lyapunov} and \prettyref{eq:topological entropy}
it is easy to see that the Margulis-Ruelle inequality particularly
implies 
\begin{equation}
h_{\textrm{top}}\left(V\right)\leq m.\Lambda_{\textrm{max}}.\label{eq:lambda max vs entropy}
\end{equation}

\section{\label{subsec:General Weyl law} Weyl Law}

We now give a proof for \prettyref{thm:general Weyl law} based on
a modification of the one in \cite[Ch. 11]{Dimassi-Sjostrand} using
wave trace asymptotics. The task being to make the argument therein
quantitative.

\subsection{Wave trace asymptotics}

First, consider the energy band $\Sigma_{\left[a-\alpha,a+\alpha\right]},$
for each $\alpha>0$. Denote by the shorthand $e^{tH_{p_{0}}^{\left[a-\alpha,a+\alpha\right]}}\coloneqq\left.e^{tH_{p_{0}}}\right|_{\Sigma_{\left[a-\alpha,a+\alpha\right]}}$
the Hamilton flow restricted to the given energy band. And by $T_{0}^{\left[a-\alpha,a+\alpha\right]}>0$
the shortest period of the restricted Hamilton flow $e^{tH_{p_{0}}^{\left[a-\alpha,a+\alpha\right]}}$.
In similar vein as \prettyref{eq:recurrence set} and \prettyref{eq:Ehrenfest time},
the recurrence sets for the energy band are defined via 
\begin{equation}
S_{T,\varepsilon}^{\left[a-\alpha,a+\alpha\right]}\coloneqq\left\{ \left(x,\xi\right)\subset\Sigma_{\left[a-\alpha,a+\alpha\right]}|\exists t\in\left[\frac{1}{2}T_{0}^{\left[a-\alpha,a+\alpha\right]},T\right]\textrm{ s.t. }d\left(e^{tH_{p_{0}}}\left(x,\xi\right),\left(x,\xi\right)\right)\leq\varepsilon\right\} ,\label{eq:band recurrence set}
\end{equation}
and $S_{T,\varepsilon}^{\left[a-\alpha,a+\alpha\right],e}\coloneqq\left\{ \left(x,\xi\right)\subset\Sigma_{\left[a-\alpha,a+\alpha\right]}|d\left(\left(x,\xi\right),S_{T,\varepsilon}^{\left[a-\alpha,a+\alpha\right]}\right)\leq\varepsilon\right\} $.
Below it shall also be useful to define the intermediete recurrence
set $S_{T,\varepsilon}^{\left[a-\alpha,a+\alpha\right],\frac{1}{2}}\coloneqq\left\{ \left(x,\xi\right)\subset\Sigma_{\left[a-\alpha,a+\alpha\right]}|d\left(\left(x,\xi\right),S_{T,\varepsilon}^{\left[a-\alpha,a+\alpha\right]}\right)\leq\frac{\varepsilon}{2}\right\} $.
Again $d$ denotes a phase space distance on the band that is equivalent
to some Riemannian distance on it. 

Furthermore

\begin{align}
T_{E}^{\ell,\left[a-\alpha,a+\alpha\right]}\left(h\right) & \coloneqq\begin{cases}
\infty; & \textrm{if }\exists C>0\textrm{ s.t.}\;\left|\partial^{\alpha}e^{tH_{p_{0}}^{\left[a-\alpha,a+\alpha\right]}}\right|\lesssim\left|t\right|^{C},\;\forall\alpha\in\mathbb{N}_{0}^{2n-1},\\
\frac{\left|\ln h\right|}{\Lambda_{\textrm{max}}^{\left[a-\alpha,a+\alpha\right]}+\ell}; & \textrm{otherwise},
\end{cases}\label{eq:Ehrenfest time-1}\\
\textrm{ and }\quad\Lambda_{\textrm{max}}^{\left[a-\alpha,a+\alpha\right]} & \coloneqq\limsup_{t\rightarrow\infty}t^{-1}\sup_{x\in\Sigma_{\left[a-\alpha,a+\alpha\right]}}\ln\left|de^{tH_{p_{0}}}\left(x\right)\right|\label{eq:max expansion rate-1}
\end{align}
are the Ehrenfest time and the maximal expansion rate respectively
of the energy band. 

Next, we choose a square microlocal partition of unity $\left\{ A_{j}=a_{j}^{W}\right\} _{j=0}^{M_{h}}\in\Psi_{h,\delta}^{0}\left(X\right)$,
$\delta\in\left[0,\frac{1}{2}\right)$, $M_{h}=O\left(h^{-2n\delta}\right)$,
$\sum_{j=0}^{M_{h}}A_{j}^{2}=1$ on $\Sigma_{\left[a-\alpha,a+\alpha\right]}\subset T^{*}X$,
as follows. The first among these is chosen such that $a_{0}\in C^{\infty}\left(T^{*}X;\left[0,1\right]\right)$,
with $A_{0}=a_{0}^{W}$ self-adjoint and 
\[
a_{0}=\begin{cases}
1, & \textrm{on }S_{T,\varepsilon}^{\left[a-\alpha,a+\alpha\right],\frac{1}{2}}\\
0, & \textrm{on }\left(S_{T,\varepsilon}^{\left[a-\alpha,a+\alpha\right],e}\right)^{c}.
\end{cases}
\]
One satisfying correct symbolic estimates in $\Psi_{h,\delta}^{0}\left(X\right)$
is found by an application of the Whitney extension theorem \cite[Sec. 2.3]{HormanderI}.
For points $p$ in the complement of $S_{T,\varepsilon}^{\left[a-\alpha,a+\alpha\right],\frac{1}{2}}$,
one has $d\left(e^{tH_{p_{0}}}p',p'\right)\geq\varepsilon$ for each
$d\left(p,p'\right)\leq\frac{\varepsilon}{4}$. It is hence covered
by open radius $\frac{\varepsilon}{4}$-balls $U_{j}$ such that $d\left(e^{tH_{p_{0}}}U_{j},U_{j}\right)\geq\frac{\varepsilon}{4}$
for $t\in\left[\frac{1}{2}T_{0}^{\left[a-\alpha,a+\alpha\right]},T\right]$.
An application of Egorov theorem to Ehrenfest time \cite[Sec. 3.4]{Nonnenmacher2010},
gives that the rest of the pseudodifferential operators $A_{1},\ldots,A_{M}\in\Psi_{h,\delta}^{0}\left(X\right)$,
$\varepsilon=ch^{\delta},$ in the partition of unity covering $\left(S_{T,\varepsilon}^{\left[a-\alpha,a+\alpha\right],e}\right)^{c}$
maybe chosen such that
\begin{align}
A_{j,t} & \coloneqq e^{-\frac{it}{h}P_{h}}A_{j}e^{\frac{it}{h}P_{h}}\in\Psi_{h,\delta'}^{0}\left(X\right)\nonumber \\
\textrm{for }\quad\delta' & =\frac{1}{2}-\left(\frac{1}{2}-\delta\right)\frac{\ell}{\Lambda_{\textrm{max}}^{\left[a-\alpha,a+\alpha\right]}+\ell}\\
\textrm{with }\quad WF\left(A_{j}\right)\cap WF\left(A_{j,t}\right) & =\emptyset,\label{eq:partition non-recurrence set}
\end{align}
 $j=1,\ldots,M$ and $t\in\left[\frac{1}{2}T_{0}^{\left[a-\alpha,a+\alpha\right]},\left(\frac{1}{2}-\delta\right)T_{E}^{\ell,\left[a-\alpha,a+\alpha\right]}\left(h\right)\right]$
.

Next choose $h-$independent functions $f,\theta\in C_{c}^{\infty}\left(\mathbb{R}\right)$
with $f,\check{\theta}=\mathcal{F}^{-1}\theta\geq0$ and $\textrm{spt}\left(\theta\right)$
contained in a sufficiently small neighbourhood of the origin. The
trace norm and trace of a positive self-adjoint operator, computed
with respect to the auxiliary density on the manifold, being equal
one has 
\begin{align}
\left\Vert A_{0}f\left(P\right)\left(\mathcal{F}_{h}^{-1}\theta\right)\left(\lambda-P\right)A_{0}\right\Vert _{\textrm{tr}} & =\textrm{tr }\left[A_{0}^{2}f\left(P\right)\left(\mathcal{F}_{h}^{-1}\theta\right)\left(\lambda-P\right)\right]\label{eq:trace norm defined}\\
 & =f\left(\lambda\right)\theta\left(0\right)\left(2\pi h\right)^{-n}\left[\int_{\Sigma_{\lambda}}a_{0}^{2}d\nu+O\left(h^{1-2\delta}\right)\right]\label{eq:wave trace asymptotics}\\
 & \leq f\left(\lambda\right)\theta\left(0\right)\left(2\pi h\right)^{-n}\left[\nu\left(S_{T,\varepsilon}^{\lambda,e}\right)+O\left(h^{1-2\delta}\right)\right]\label{eq:tr equals tr norm}
\end{align}
for each $\lambda\in\mathbb{R}$. Here the asymptotics of the trace
in the line \prettyref{eq:wave trace asymptotics} above are evaluated
by a standard FIO parametrix and application of stationary phase formula
as in \cite[Ch. 10]{Dimassi-Sjostrand}. The exponent in $h^{1-2\delta}$
arises due to the presence of two derivatives of the amplitude $a_{0}^{2}\in S_{\delta}^{0}$
with each $h$-term in the stationary phase formula. The last equation
\prettyref{eq:tr equals tr norm} can be claimed for arbitrary $\theta\in C_{c}^{\infty}\left(\mathbb{R}\right)$.
The condition $\check{\theta}\geq0$ can be removed as in \cite[eq. 11.5]{Dimassi-Sjostrand}. 

Next, for $\theta_{c}\left(t\right)\coloneqq\theta\left(t-c\right)$
one has $\mathcal{F}_{h}^{-1}\theta_{c}\left(x\right)=e^{i\frac{xc}{h}}\mathcal{F}_{h}^{-1}\theta\left(x\right)$.
Furthermore $e^{ic\left(\lambda-P\right)}$ being a unitary operator,
the trace norm on the left hand side of \prettyref{eq:trace norm defined}
is hence unchanged under translation of $\theta$. By writing an arbitrary
$\theta\in C_{c}^{\infty}\left(-T,T\right)$, of possibly $h$-dependent
compact support $T=T\left(h\right)$, as a sum of $O\left(T\right)$
translates of functions with $h$-independent compact support near
the origin we obtain 
\[
\left\Vert A_{0}f\left(P\right)\left(\mathcal{F}_{h}^{-1}\theta\right)\left(\lambda-P\right)A_{0}\right\Vert _{\textrm{tr}}=f\left(\lambda\right)\left\Vert \theta\right\Vert _{C^{0}}\left(2\pi h\right)^{-n}O\left(T\left[\nu\left(S_{T,\varepsilon}^{\lambda,e}\right)+O\left(h^{1-2\delta}\right)\right]\right).
\]
Furthermore, Egorov theorem to Ehrenfest time gives 
\begin{equation}
\textrm{tr }\left[A_{j}^{2}e^{\frac{it}{h}P}\right]=\textrm{tr }\left[e^{\frac{it}{h}P}A_{j,t}A_{j}\right]=O\left(h^{\infty}\right)\label{eq:Egorov consequence}
\end{equation}
for $\theta\in C_{c}^{\infty}\left(\frac{1}{2}T_{0}^{\left[a-\alpha,a+\alpha\right]},T\right)$,
$T\leq\left(\frac{1}{2}-\delta\right)T_{E}^{\ell,\left[a-\alpha,a+\alpha\right]}\left(h\right)$
following \prettyref{eq:partition non-recurrence set}. Now note that
$A_{j}$'s were chosen to comprise of a square partition of unity
on the band $\Sigma_{\left[a-\alpha,a+\alpha\right]}$. While $WF\left(f\left(P\right)\right)\subset\Sigma_{\left[a-\alpha,a+\alpha\right]}$
for $f\in C_{c}^{\infty}\left(a-\alpha,a+\alpha\right)$ by \prettyref{eq:wavefront function of operator}.
Thus the last two equations combine to give
\begin{equation}
\left|\textrm{tr }\left[f\left(P\right)\left(\mathcal{F}_{h}^{-1}\theta\right)\left(\lambda-P\right)\right]\right|=f\left(\lambda\right)\left\Vert \theta\right\Vert _{C^{0}}\left(2\pi h\right)^{-n}O\left(T\left[\nu\left(S_{T,\varepsilon}^{\lambda,e}\right)+O\left(h^{1-2\delta}\right)\right]\right)\label{eq:wave trace estimate large support}
\end{equation}
for $f\in C_{c}^{\infty}\left(a-\alpha,a+\alpha\right)$ and $\theta\in C_{c}^{\infty}\left(-T,T\right)$
with $T\leq\left(\frac{1}{2}-\delta\right)T_{E}^{\ell,\left[a-\alpha,a+\alpha\right]}\left(h\right)$.

Finally we write an arbitrary $\theta\in C_{c}^{\infty}\left(-T,T\right)$
as a sum $\theta=\theta_{1}+\theta_{2}$ of an $h$-independent function
$\theta_{1}$ supported sufficiently near the origin and $\theta_{2}$
supported away from the origin. Applying the estimate \prettyref{eq:wave trace estimate large support}
for $\theta_{2}$ while using again the FIO parametrix for the $h$-independent
function $\theta_{1}$ gives
\begin{align}
\textrm{tr }\left[f\left(P\right)\left(\mathcal{F}_{h}^{-1}\theta\right)\left(\lambda-P\right)\right] & =f\left(\lambda\right)\theta\left(0\right)\left(2\pi h\right)^{-n}\left[\int_{\Sigma_{\lambda}}d\nu+\left\Vert \theta\right\Vert _{C^{0}}O\left(T\nu\left(S_{T,\varepsilon}^{\lambda,e}\right)+Th^{1-2\delta}\right)\right]\label{eq:starting Taberian}
\end{align}
for $f\in C_{c}^{\infty}\left(a-\alpha,a+\alpha\right)$ and $\theta\in C_{c}^{\infty}\left(-T,T\right)$
with $T\leq\left(\frac{1}{2}-\delta\right)T_{E}^{\ell,\left[a-\alpha,a+\alpha\right]}\left(h\right)$.

\subsection{Tauberian argument}

Following the last equation \prettyref{eq:starting Taberian}, the
rest of the proof of \prettyref{thm:general Weyl law} follows a standard
Tauberian argument as in \cite[Ch. 11]{Dimassi-Sjostrand} or \cite[Appx. B]{Safarov-Vassiliev-1997}.
We provide the details below for completeness.

First choose an even, $h-$independent Schwartz function $\theta\in\mathcal{S}\left(\mathbb{R}\right)$
such that $\check{\theta}\geq\frac{1}{1+\epsilon}$, $\epsilon>0$,
on $\left[0,1\right]$ and $1=\theta\left(0\right)=\int d\xi\check{\theta}\left(\xi\right)$.
Seting $\theta_{T}\left(x\right)=\theta\left(T^{-1}x\right)$, satisfying
$\check{\theta_{T}}\left(\xi\right)=T\check{\theta}\left(T\xi\right)$,
and choosing $f\in C_{c}^{\infty}\left(a-\alpha,a+\alpha\right)$
with $f\left(\lambda\right)=1$, the trace expansion \prettyref{eq:starting Taberian}
now gives 
\begin{align*}
\frac{T}{\left(1+\epsilon\right)h}N\left(\lambda,\lambda+T^{-1}h\right)\leq & \textrm{tr }\left[f\left(P\right)\left(\mathcal{F}_{h}^{-1}\theta_{T}\right)\left(\lambda-P\right)\right]\\
= & \left(2\pi h\right)^{-n}\left[\int_{\Sigma_{\lambda}}d\nu+O\left(T\nu\left(S_{T,\varepsilon}^{\lambda,e}\right)+Th^{1-2\delta}\right)\right]
\end{align*}
 for each $\epsilon>0$, $\lambda\in\left[a-\alpha,a+\alpha\right]$
and $T\leq\left(\frac{1}{2}-\delta\right)T_{E}^{\ell,\left[a-\alpha,a+\alpha\right]}\left(h\right)$.
And hence 
\begin{equation}
N\left(\lambda,\lambda+T^{-1}h\right)\leq\left(2\pi h\right)^{-n}\left[T^{-1}\int_{\Sigma_{\lambda}}d\nu+O\left(\nu\left(S_{T,\varepsilon}^{\lambda,e}\right)+h^{1-2\delta}\right)\right]\label{eq:local Weyl law}
\end{equation}
 for each $\ell>0$, $\lambda\in\left[a-\alpha,a+\alpha\right]$ and
$T\leq\left(\frac{1}{2}-\delta\right)T_{E}^{\ell,\left[a-\alpha,a+\alpha\right]}\left(h\right)$.

Next, define the spectral measure for $P$ is defined as $\mu_{f}\left(\lambda'\right)\coloneqq\sum_{\lambda\in\textrm{Spec}\left(P\right)}f\left(\lambda\right)\delta\left(\lambda-\lambda'\right)$.
Now choose a different even function $\theta\in\mathcal{S}\left(\mathbb{R}\right)$
such that its transform satisfies $\textrm{spt}\left(\check{\theta}\right)\subset\left[-1,1\right]$,
$1\geq\check{\theta}\left(\xi\right)\geq0$ and $\int\check{\theta}\left(\xi\right)d\xi=1$.
The two term asymptotics of the wave trace from \cite[Ch. 10]{Dimassi-Sjostrand}
now gives 
\begin{align*}
\mu_{f}\ast\left(\mathcal{F}_{h}^{-1}\theta\right)\left(\lambda\right) & =f\left(\lambda\right)\theta\left(0\right)\left(2\pi h\right)^{-n}\left[\underbrace{\int_{\Sigma_{\lambda}}d\nu}_{\eqqcolon c_{0}\left(\lambda\right)}+h\underbrace{\int_{\Sigma_{\lambda}}p_{1}d\nu}_{\eqqcolon c_{1}\left(\lambda\right)}+O\left(h^{2}\right)\right],
\end{align*}
$\forall\lambda\in\mathbb{R}$, with the second term involving the
sub-principal symbol $p_{1}$ of the operator $P$.

Both sides above involving Schwartz functions in $\lambda$, the remainder
above can be replaced by $O\left(\frac{h^{2}}{\left\langle \lambda\right\rangle ^{2}}\right)$.
Integrating further gives
\begin{align}
 & \int_{-\infty}^{a}d\lambda\int d\lambda'\left(\mathcal{F}_{h}^{-1}\theta\right)\left(\lambda-\lambda'\right)\mu_{f}\left(\lambda'\right)\label{eq:main trace exp integrated}\\
= & \int_{-\infty}^{0}d\lambda\int d\lambda''\left(\mathcal{F}_{h}^{-1}\theta\right)\left(\lambda-\lambda''\right)\mu_{f}\left(\lambda''+a\right)\\
= & \theta\left(0\right)\left(2\pi h\right)^{-n}\left[\int_{-\infty}^{a}d\lambda f\left(\lambda\right)c_{0}\left(\lambda\right)+h\int_{-\infty}^{a}d\lambda f\left(\lambda\right)c_{1}\left(\lambda\right)+O\left(h^{2}\right)\right].\nonumber 
\end{align}
Now note that
\begin{equation}
\int_{-\infty}^{0}d\lambda\left(\mathcal{F}_{h}^{-1}\theta\right)\left(\lambda-\lambda''\right)=1_{\left(-\infty,0\right]}\left(\lambda''\right)+\phi_{0}\left(\frac{\lambda''}{h}\right)\label{eq: definition remainder}
\end{equation}
where $\phi_{0}\left(x\right)\coloneqq\int_{-\infty}^{0}dt\check{\theta}\left(t-x\right)-1_{\left(-\infty,0\right]}\left(x\right)$
is a function that is rapidly decaying with all derivatives, odd,
smooth on $\mathbb{R}_{x}\setminus\left\{ 0\right\} $ and satisfies
$\phi_{0}'\left(x\right)=\check{\theta}\left(-x\right)$ for $x\neq0$. 

Next with $x\geq0$ we compute 
\begin{align}
 & \left|\phi_{0}\left(x\right)-\phi_{0}\ast\check{\theta}_{T}\left(x\right)\right|\nonumber \\
= & \left|\int dy\left[\phi_{0}\left(x\right)-\phi_{0}\left(x-T^{-1}y\right)\right]\check{\theta}\left(y\right)\right|\nonumber \\
\leq & \int_{y\leq xT}dy\left|\phi'_{0}\left(c\left(x,y\right)\right)\right|T^{-1}\left|y\right|\check{\theta}\left(y\right)+2\int_{y\geq xT}dy\check{\theta}\left(y\right)\nonumber \\
\leq & T^{-1}\underbrace{\int_{-\infty}^{xT}dy\left|y\right|\check{\theta}\left(y\right)}_{=\theta_{1}\left(xT\right)}+2\underbrace{\int_{y\geq xT}dy\check{\theta}\left(y\right)}_{=\theta_{2}\left(xT\right)}\label{eq:replace with convolution}
\end{align}
where $c\left(x,y\right)\in\left[x-T^{-1}y,x\right]$. A similar estimate
holds for $x\leq0$.

Now pairing the second term of \prettyref{eq:replace with convolution}
with $\mu_{f}\left(\lambda''+a\right)$ gives 
\begin{align}
\int d\lambda''\theta_{2}\left(\frac{\lambda''T}{h}\right)\mu_{f}\left(\lambda''+a\right)\leq & \left(2\pi h\right)^{-n}\left[T^{-1}\left\Vert f\right\Vert _{C^{0}}\left(\int_{\Sigma_{a}}d\nu\right)+O\left(\nu\left(S_{T,\varepsilon}^{a,e}\right)+h^{1-2\delta}\right)\right]\label{eq:pairing with first term}
\end{align}
on covering $\mathbb{R}_{\lambda''}$ with intervals of size $O\left(T^{-1}h\right)$
and using the Weyl estimate \prettyref{eq:local Weyl law}. A similar
estimate 
\begin{equation}
\int d\lambda''T^{-1}\theta_{1}\left(\frac{\lambda''T}{h}\right)\mu_{f}\left(\lambda''+a\right)=O\left(h^{-n}T^{-1}\left[T^{-1}+\nu\left(S_{T,\varepsilon}^{a,e}\right)+h^{1-2\delta}\right]\right)\label{eq: second term}
\end{equation}
then gives 
\begin{align}
 & \int d\lambda''\left[\phi_{0}\left(\frac{\lambda''}{h}\right)-\phi_{0}\ast\check{\theta}_{T}\left(\frac{\lambda''}{h}\right)\right]\mu_{f}\left(\lambda''+a\right)\label{eq:cutoff vs cutoff conv}\\
\leq & \left(2\pi h\right)^{-n}\left[T^{-1}\left\Vert f\right\Vert _{C^{0}}\left(\int_{\Sigma_{a}}d\nu\right)+O\left(T^{-2}+\nu\left(S_{T,\varepsilon}^{a,e}\right)+h^{1-2\delta}\right)\right].\nonumber 
\end{align}
on combining\prettyref{eq:replace with convolution}, \prettyref{eq:pairing with first term}
and \prettyref{eq: second term}. 

The second term above \prettyref{eq:cutoff vs cutoff conv} is estimated
on integrating \prettyref{eq:starting Taberian} against $\phi_{0}$
as
\begin{align}
\int d\lambda''\phi_{0}\ast\check{\theta}_{T}\left(\frac{\lambda''}{h}\right)\mu_{f}\left(\lambda''+a\right) & =\left(2\pi h\right)^{-n}\left[\intop d\lambda\phi_{0}\left(\lambda\right)f\left(0\right)\theta\left(0\right)c_{0}\left(0\right)+O\left(\nu\left(S_{T,\varepsilon}^{\lambda,e}\right)+h^{1-2\delta}\right)\right]\nonumber \\
 & =O\left(\nu\left(S_{T,\varepsilon}^{\lambda,e}\right)h^{-n}+h^{-n+1-2\delta}\right).\label{eq: smoothed function expansion}
\end{align}
since $\phi_{0}$ is an odd function.

Finally, set $f_{a}^{-}=1_{\left(-\infty,a\right]}\left(x\right)f\left(x\right)$.
Then combining \prettyref{eq:main trace exp integrated}, \prettyref{eq: definition remainder},
\prettyref{eq:cutoff vs cutoff conv} and \prettyref{eq: smoothed function expansion}
gives 
\begin{align*}
\textrm{tr }f_{a}^{-}\left(P\right) & =\int d\lambda'1_{\left(-\infty,a\right]}\left(\lambda'\right)\mu_{f}\left(\lambda'\right)\\
 & =\left(2\pi h\right)^{-n}\left(\int_{-\infty}^{a}d\lambda f\left(\lambda\right)c_{0}\left(\lambda\right)+h\int_{-\infty}^{a}d\lambda f\left(\lambda\right)c_{1}\left(\lambda\right)+O\left(h^{2}\right)\right)\\
 & \qquad+\int d\lambda''\phi_{0}\left(\frac{\lambda''}{\sqrt{h}}\right)\mu_{f}\left(\lambda''+a\right)\\
 & =\left(2\pi h\right)^{-n}\left(\int_{-\infty}^{a}d\lambda f\left(\lambda\right)c_{0}\left(\lambda\right)+h\int_{-\infty}^{a}d\lambda f\left(\lambda\right)c_{1}\left(\lambda\right)\right)+R\left(h\right),\\
\textrm{with }\quad R\left(h\right) & \leq\left(2\pi h\right)^{-n+1}\left[T^{-1}\left\Vert f\right\Vert _{C^{0}}\left(\int_{\Sigma_{a}}d\nu\right)+O\left(T^{-2}+\nu\left(S_{T,\varepsilon}^{a,e}\right)+h^{1-2\delta}\right)\right],
\end{align*}
for each $\ell>0$ and $T\leq\left(\frac{1}{2}-\delta\right)T_{E}^{\ell,\left[a-\alpha,a+\alpha\right]}\left(h\right)$.
Letting $\alpha\rightarrow0$ and using the upper semi-continuity
of $\Lambda_{\textrm{max}}$, one obtains the above for each $\ell>0$
and $T\leq\left(\frac{1}{2}-\delta\right)T_{E}^{\ell,a}\left(h\right)$.

A similar estimate can be proved for the functional trace of $f_{b}^{+}\left(x\right)=1_{\left[b,\infty\right)}\left(x\right)f\left(x\right)$.
The Weyl law of \prettyref{thm:general Weyl law} now follows on writing
the counting function as a difference of two such functional traces.

\section{\label{sec:Examples-of-recurrence} Examples of recurrence}

In this section we estimates on the volumes of recurrence sets of
various flows. These shall be used in the next \prettyref{sec:Proofs-of-the-corollaries}
to prove the corollaries stated in the introduction.

\subsection{\label{subsec:Anosov-flows} Anosov flows}

First we shall consider the recurrence set of an Anosov vector field
$V$ on compact manifold $Y^{m}$ of dimension $m$. To recall, a
vector field $V$ is said to be Anosov if exist a constant $c_{1}>0$
and a continuous splitting
\begin{align}
TY & =\mathbb{R}\left[V\right]\oplus E^{u}\oplus E^{s}\quad\textrm{ such that}\nonumber \\
\left\Vert \left.e^{tV}\right|_{E^{s}}\right\Vert  & \leq e^{-c_{1}t},\nonumber \\
\left\Vert \left.e^{-tV}\right|_{E^{u}}\right\Vert  & \leq e^{-c_{1}t},\label{eq:Anosov condition}
\end{align}
$\forall t>0,$ that is invariant under the flow of $V$. Here the
norm is taken with respect to some Riemannian metric $g^{TY}$ on
the manifold. For such a vector field its recurrence set $S_{T,\varepsilon}$
\prettyref{eq:recurrence set}, Lyapunov exponents $\lambda_{j}\left(y\right)$
\prettyref{eq:Lyapunoc exponents}, maximal expansion rate $\Lambda_{\textrm{max}}$
\prettyref{eq:lambda max as max antropy} and topological entropy
$\mathtt{h}_{\textrm{top}}$ \prettyref{eq:topological entropy} can
be analogously defined.

Let $\lambda>\Lambda_{\textrm{max}}$ be greater than the maximal
expansion rate. By definition of the maximal expansion rate, and the
semi-group property of the flow, there exists $c>0$ such that
\begin{align}
\left\Vert \left(e^{tV}\right)^{*}f\right\Vert _{C^{2}} & \leq ce^{\lambda t}\left\Vert f\right\Vert _{C^{2}}\quad\textrm{ and }\nonumber \\
d^{g^{TY}}\left(e^{tV}x_{1},e^{tV}x_{2}\right) & \leq ce^{\lambda t}d^{g^{TY}}\left(x_{1},x_{2}\right)\label{eq:growth of distance}
\end{align}
for all $t>0$, $x_{1},x_{2}\in Y$ and $f\in C^{\infty}\left(Y\right)$.
From here exponential bounds on the volume of the recurrence set
\begin{align}
\nu\left(S_{T,\varepsilon}\right) & =O\left(\varepsilon^{m}e^{m\lambda T}\right)\label{eq:recurrence measure Anosov}\\
\nu\left(S_{T,\varepsilon}^{e}\right) & =O\left(\varepsilon^{m}e^{m\lambda T}\right),\label{eq:recurrence measure extended Anosov}
\end{align}
can be proved following an argument given in \cite{Dyatlov-Zworski2016}.
Namely, the recurrence set above has an obvious lift 
\begin{equation}
\tilde{S}_{T,\varepsilon}\coloneqq\left\{ \left(y,t\right)|t\in\left[\frac{1}{2}T_{0},T\right]\textrm{ s.t. }d^{g^{TY}}\left(e^{tV}x,x\right)\leq\varepsilon\right\} \subset Y\times\mathbb{R}\label{eq:lift of the recurrence set}
\end{equation}
satisfying $\pi_{Y}\left(\tilde{S}_{T,\varepsilon}\right)=S_{T,\varepsilon}$
under the projection onto the first $Y$ factor. 

The volume bounds \prettyref{eq:recurrence measure Anosov}, \prettyref{eq:recurrence measure extended Anosov}
then follow from the following proposition.
\begin{prop}
\label{prop:Anosov lifted recurrence bound} For each $\lambda>\Lambda_{\textrm{max}}$
, the lift $\tilde{S}_{T,\varepsilon}$ \prettyref{eq:lift of the recurrence set}
satisfies the volume estimate
\begin{equation}
\nu_{Y\times\mathbb{R}}\left(\tilde{S}_{T,\varepsilon}\right)=O\left(\varepsilon^{m}e^{m\lambda T}\right)\label{eq:measure estimate on lift}
\end{equation}
with respect to the Riemannian product measure on $Y\times\mathbb{R}$.
\end{prop}

\begin{proof}
First we claim that there exist $C,\delta>0$ of the following significance:
for each $\varepsilon>0$ and each pair $\left(x,t\right),\left(x',t'\right)\in\tilde{S}_{T,\varepsilon}$
satisfying $\left|t-t'\right|\leq\delta$, $d^{g^{TY}}\left(x,x'\right)\leq\delta e^{-\lambda t}$
one has 
\begin{equation}
\left|t-t'\right|\leq C\varepsilon,\quad d^{g^{TY}}\left(x,\cup_{t\in\left[-1,1\right]}e^{tV}x'\right)\leq C\varepsilon.\label{eq:orbits are close}
\end{equation}
By choosing $\delta$ sufficiently small and using \prettyref{eq:growth of distance}
we work in a sufficiently small geodesic coordinate chart. The $V$-direction
and $E^{u}\oplus E^{s}$ being transverse, we may replace $x'$ by
$e^{tV}\left(x'\right)$, $t\in\left[-1,1\right]$, to arrange $x-x'\in E^{u}\left(x\right)\oplus E^{s}\left(x\right)$.
Using \prettyref{eq:growth of distance} and a Taylor expansion in
$x,t$ we obtain 
\begin{align*}
\left|e^{tV}\left(x\right)-e^{t'V}\left(x'\right)-de^{tV}\left(x\right)\left(x-x'\right)-V\left(e^{tV}\left(x'\right)\right)\left(t-t'\right)\right| & \leq c_{3}e^{\lambda t}\left|x-x'\right|^{2}+c_{3}\left|t-t'\right|^{2}\\
 & \leq c_{3}\delta\left|x-x'\right|+c_{3}\delta\left|t-t'\right|.
\end{align*}
Since $\left(x,t\right),\left(x',t'\right)\in\tilde{S}_{T,\varepsilon}$
the above gives 
\begin{align*}
c_{3}\delta\left|x-x'\right|+c_{3}\delta\left|t-t'\right|+2\varepsilon & \geq\left|\left(I-de^{tV}\left(x\right)\right)\left(x-x'\right)-V\left(e^{tV}\left(x'\right)\right)\left(t-t'\right)\right|\\
 & \geq c_{4}\left|x-x'\right|+c_{4}\left|t-t'\right|
\end{align*}
with the second line above following from the Anosov property. It
then remains to choose $\delta$ sufficiently small in relation to
$c_{3},c_{4}$ to obtain \prettyref{eq:orbits are close}.

Finally, let $x_{j}$, $j=1,\ldots N$, be a maximal set of points
such that $d^{g^{TY}}\left(x_{i},x_{j}\right)\geq\delta e^{-\lambda T}$.
As the balls $\left\{ B_{\frac{\delta e^{-\lambda T}}{2}}\left(x_{j}\right)\right\} _{j=1}^{N}$
centered at these points are disjoint, the bound $N\leq c_{5}e^{m\lambda T}$
follows by a computation of the total volume. Furthermore the sets
\begin{align*}
B_{j,k} & \coloneqq B_{2\delta e^{-\lambda T}}\left(x_{j}\right)\times\left[\frac{1}{2}T_{0}+k\delta,\frac{1}{2}T_{0}+\left(k+1\right)\delta\right],\\
S_{j,k} & \coloneqq\tilde{S}_{T,\varepsilon}\cap B_{j,k},\qquad j=1,\ldots N,\quad k=0,\ldots,1+\left[\delta^{-1}T\right],
\end{align*}
cover $Y\times\left[\frac{1}{2}T_{0},T\right]$ and $\tilde{S}_{T,\varepsilon}$
respectively. By \prettyref{eq:orbits are close} small $O\left(\varepsilon\right)$
size neighborhoods of the orbits
\[
\left(\underbrace{\frac{1}{2}T_{0}+\left(k+\frac{1}{2}\right)\delta}_{\eqqcolon t_{k}},\cup_{t\in\left[-1,1\right]}e^{\left(t_{k}+t\right)V}\left(x_{j}\right)\right)
\]
 of volume $O\left(\varepsilon^{m}\right)$, then cover $\tilde{S}_{T,\varepsilon}$
proving \prettyref{eq:measure estimate on lift}.
\end{proof}

\subsubsection{$\Lambda_{\textrm{max}}$ vs $\mathtt{h}_{\textrm{top}}$}

In section \prettyref{subsec:Dynamical-invariants} we stated how
the inequality $h_{\textrm{top}}\left(V\right)\leq m.\Lambda_{\textrm{max}}$
\prettyref{eq:lambda max vs entropy} follows from the Margulis-Ruelle
formula for a general flow of a vector field $V$ on a compact manifold
$m$-dimensional manifold $Y$. In this section we show that a reverse
inequality holds between the two invariants when the vector field
is further assumed to be Anosov. Namely, we shall prove the following.
\begin{thm}
\label{thm: lambda max vs htop thm.} For an Anosov vector field $V$
on a compact $m$-dimensional manifold $Y$ one has 
\begin{equation}
\frac{m}{4}.\Lambda_{\textrm{max}}\leq h_{\textrm{top}}\left(V\right).\label{eq: lambda mx leq htop for anosov}
\end{equation}
\end{thm}

We shall prove the above at the end of this subsection following some
preparation. Namely, to prove the above we shall use the equivalent
Bowen-Margulis definition of topological entropy. For each $T>0$
one defines the Bowen distance on $Y$ via 
\[
d_{T}^{g^{TY}}\left(y_{1},y_{2}\right)\coloneqq\sup_{t\in\left[0,T\right]}d^{g^{TY}}\left(e^{tR}y_{1},e^{tR}y_{2}\right),
\]
where $d^{g^{TY}}$ denotes the Riemannian distance corresponding
to some Riemannian metric $g^{TY}$ on $Y$. A $\left(T,\epsilon\right)$
separated subset $S\subset Y$ is a finite set in which any two distinct
points are at least distance $\epsilon$ apart with respect to the
above $d_{T}$. Denote by $N\left(T,\epsilon\right)$ the maximum
cardinality of a $\left(T,\epsilon\right)$ separated set in $Y$.
The topological entropy of the flow \prettyref{eq:topological entropy}
is now equivalently defined by 
\begin{equation}
\mathtt{h}_{\textrm{top}}=\mathtt{h}_{\textrm{top}}\left(V\right)\coloneqq\lim_{\epsilon\rightarrow0}\left(\limsup_{T\rightarrow\infty}\frac{\ln N\left(T,\epsilon\right)}{T}\right).\label{eq:topological entropy-1}
\end{equation}

Next for each $s\in\left(0,1\right]$, let $\mathcal{D}_{s}\left(Y\right)$,
$\mathcal{D}_{s-}\left(Y\right)$ respectively be the set of compatible
distorted distance functions $d$ on the manifold satisfying 
\begin{align*}
d^{g^{TY}} & \lesssim d\lesssim\left(d^{g^{TY}}\right)^{s}\\
d^{g^{TY}} & \lesssim d\lesssim\left(d^{g^{TY}}\right)^{s-\epsilon},\quad\textrm{ for some }\epsilon>0,
\end{align*}
respectively. The following inclusions are clear
\begin{align*}
\mathcal{D}_{s}\left(Y\right) & \subset\mathcal{D}_{s'}\left(Y\right),\\
\mathcal{D}_{s-}\left(Y\right) & \subset\mathcal{D}_{s'-}\left(Y\right),\quad s'<s.
\end{align*}
Furthermore, all distances in $\mathcal{D}_{s}\left(Y\right)$, $\mathcal{D}_{s-}\left(Y\right)$
define the same manifold topology while $\mathcal{D}_{1}\left(Y\right)$
is the set of all distances equivalent to the $d^{g^{TY}}$ and hence
includes all Riemannian distances. To each distance $d\in\mathcal{D}_{s}\left(Y\right),\,\mathcal{D}_{s-}\left(Y\right)$
we can associate the Lipschitz constant of its time one flow
\begin{equation}
L_{d}=L_{d}\left(e^{V}\right)\coloneqq\sup_{x_{1}\neq x_{2}}\frac{d\left(e^{V}y_{1},e^{V}y_{2}\right)}{d\left(y_{1},y_{2}\right)}.\label{eq:Lipschitz constant}
\end{equation}
The following notion of the local skewness of the time one map shall
also be useful. It is defined as
\[
SL_{d}\left(e^{V}\right)\coloneqq\sup_{\varepsilon>0}\inf_{0<d\left(x_{1},x_{2}\right)<\varepsilon}\frac{d\left(e^{V}y_{1},e^{V}y_{2}\right)}{d\left(y_{1},y_{2}\right)}.
\]

We now have the following inequalities for topological entropy.
\begin{lem}
\label{lem:Entropy inequality} The topological entropy \prettyref{eq:topological entropy}
of an Anosov vector field satisfies the inequalities
\[
\frac{m}{2}\left(\inf_{d\in\mathcal{D}_{\frac{1}{2}-}\left(Y\right)}\ln L_{d}\right)\leq\mathtt{h}_{\textrm{top}}\leq m\left(\inf_{d\in\mathcal{D}_{\frac{1}{2}-}\left(Y\right)}\ln L_{d}\right)
\]
in relation to the infimum of the log Lipschitz constants \prettyref{eq:Lipschitz constant}
in $\mathcal{D}_{\frac{1}{2}-}\left(Y\right)$.
\end{lem}

\begin{proof}
Let $\textrm{HD}\left(d\right)$ denote the Hausdorff dimension of
the manifold with respect to the distance $d\in\mathcal{D}_{\frac{1}{2}-}\left(Y\right)$.
The inequalities 
\begin{equation}
\textrm{HD}\left(d\right)\ln SL_{d}\leq\mathtt{h}_{\textrm{top}}\leq\textrm{HD}\left(d\right)\ln L_{d}\label{eq:SL vs htop vs L}
\end{equation}
are fairly well known (see \cite{Fathi89,Roth2020} or \cite[Thm. 7.15]{Walters-book-82}).
Furthermore $\frac{m}{2}\leq\textrm{HD}\left(d\right)\leq m$ follows
from the definition. This proves one half of the lemma 
\begin{equation}
\mathtt{h}_{\textrm{top}}\leq m\left(\inf_{d\in\mathcal{D}_{\frac{1}{2}-}\left(Y\right)}\ln L_{d}\right).\label{eq:comparison of flow invariants}
\end{equation}

One is now left with constructing a sequence of distances $d_{k}\in\mathcal{D}_{\frac{1}{2}-}\left(Y\right)$,
$k=1,2,\ldots$ such that $m\,\ln L_{d_{k}}$ approaches $2\mathtt{h}_{\textrm{top}}$
as $k\rightarrow\infty$. Such a sequence of distances $d_{k}$ can
be constructed for expansive maps \cite{Fathi89,Roth2020}, cf. also
the construction of the Hamenstädt distance \cite{Hamenstadt89}.
The time one map $e^{V}$ is however not expansive in the flow direction.
This lack of expansiveness can nonetheless be replaced with the following
instability property that is satisfied by the flow $e^{V}$ \cite{Norton-Obrien73}:
there is a positive constant $c_{2}>0$ for which one has the following
implication 
\begin{equation}
y\neq e^{tV}x,\,\forall t\in\mathbb{R}\implies d^{g^{TY}}\left(e^{jV}x,e^{jV}y\right)>c_{2}\;\textrm{ for some }j\in\mathbb{Z}.\label{eq:instability of Anosov flows}
\end{equation}
In fact the proof of the above therein gives the following stronger
statement: for any $\epsilon>0$ there exist positive constants $c>0$,
$\alpha>1$ such that for $\alpha_{\epsilon}\coloneqq\alpha+\epsilon$
one has the stronger implication 
\begin{align*}
 & y\neq e^{tV}x,\,\forall t\in\mathbb{R},\;d^{g^{TY}}\left(x,y\right)<c_{2},\qquad\\
\implies & \alpha d^{g^{TY}}\left(x,y\right)\leq\max\left\{ d^{g^{TY}}\left(e^{V}x,e^{V}y\right),d^{g^{TY}}\left(e^{-V}x,e^{-V}y\right)\right\} \leq\alpha_{\epsilon}d^{g^{TY}}\left(x,y\right).
\end{align*}
The constant $\alpha$ can be related to the exponent $c_{1}$ in
the definition \prettyref{eq:Anosov condition} of the Anosov condition.

We now define 
\[
N\left(x,y\right)\coloneqq\begin{cases}
\infty, & x=y,\\
\inf\left\{ N\in\mathbb{N}_{0}|\max_{j\in\left[-N,N\right]}d^{g^{TY}}\left(e^{jV}x,e^{jV}y\right)>c_{2}\alpha^{-\left|j\right|}\right\} , & x\neq y.
\end{cases}
\]
The following bounds are straightforward 
\begin{equation}
\max\left\{ 0,\frac{\ln\frac{c_{2}}{d\left(x,y\right)}}{\ln\alpha\alpha_{\epsilon}}\right\} \leq N\left(x,y\right)\leq\max\left\{ 0,\frac{\ln\frac{c_{2}}{d\left(x,y\right)}}{\ln\alpha}\right\} .\label{eq:bounds on N}
\end{equation}
And we now set
\begin{align}
\rho\left(x,y\right)\coloneqq & \alpha^{-N\left(x,y\right)},\quad\textrm{ satisfying }\label{eq:choice of alpha}\\
\frac{d^{g^{TY}}\left(x,y\right)}{c_{2}}\leq\rho\left(x,y\right) & \leq\left[\frac{d^{g^{TY}}\left(x,y\right)}{c_{2}}\right]^{\ln\alpha/\ln\left(\alpha\alpha_{\epsilon}\right)}\quad\textrm{ for }d^{g^{TY}}\left(x,y\right)\leq c_{2}.\label{eq:d vs rho}
\end{align}
It hence follows that $\rho$ defines the same manifold topology as
$d^{g^{TY}}$, but it does not quite define a distance. The inequalities
\prettyref{eq:bounds on N} further give $d^{g^{TY}}\left(x,y\right)\geq\frac{c_{2}}{2}\implies N\left(x,y\right)\leq\frac{\ln2}{\ln\alpha}\implies\alpha^{N}\leq\alpha^{\frac{\ln2}{\ln\alpha}}=2$.
And applying the triangle inequality for $d^{g^{TY}}$ one obtains
\begin{align*}
\min\left\{ N\left(x,y\right),N\left(y,z\right)\right\}  & \leq M+N\left(x,z\right)\quad\textrm{ and }\\
\quad\rho\left(x,z\right) & \leq2\max\left\{ \rho\left(x,y\right),\rho\left(y,z\right)\right\} \quad\forall x,y,z\in Y
\end{align*}
as a weaker version of the triangle inequality for $\rho$. One now
applies Frink's metrization theorem to obtain the existence of a metric
$D$ on $Y$ satisfying 
\begin{equation}
D\left(x,y\right)\leq\rho\left(x,y\right)\leq4D\left(x,y\right).\label{eq:D vs rho comparison}
\end{equation}
Thus $D$ defines the same topology as $d^{g^{TY}}$. And furthermore
we have $D\in\mathcal{D}_{\frac{1}{2}-}\left(Y\right)$ on account
of \prettyref{eq:d vs rho} and \prettyref{eq:D vs rho comparison}.

It is now an exercise to show that $\rho\left(e^{jV}x,e^{jV}y\right)\leq\alpha^{j}\rho\left(x,y\right)$
with equality on some neighborhood $V_{j}\subset Y\times Y$ of the
diagonal in the product. From \prettyref{eq:D vs rho comparison}
this gives 
\begin{align}
D\left(e^{jV}x,e^{jV}y\right) & \leq4\alpha^{j}D\left(x,y\right),\quad\forall x,y\in Y,\nonumber \\
D\left(e^{jV}x,e^{jV}y\right) & \geq\frac{1}{4}\alpha^{j}D\left(x,y\right),\quad\forall\left(x,y\right)\in V_{j}.\label{eq:Lip =000026 SL est for D}
\end{align}
And thus 
\begin{equation}
L_{D}\left(e^{jV}\right)\leq4\alpha^{j}\leq16SL_{D}\left(e^{jV}\right).\label{eq: L vs SL for D}
\end{equation}
Finally we define the following sequence of distances
\begin{equation}
d_{k}\left(x,y\right)\coloneqq\max_{0\leq j\leq k-1}\frac{D\left(e^{jV}x,e^{jV}y\right)}{L_{D}^{j/n}}\label{eq: approximating metrics}
\end{equation}
$k\in\mathbb{N}$, which are all equivalent to $D$. Their Lipschitz
constants $L_{d_{k}}\left(e^{V}\right)=\left[L_{D}\left(e^{kV}\right)\right]^{1/k}$
are seen to be given in terms of the $D$-Lipschitz constants of the
time $k$ map. Using \prettyref{eq:SL vs htop vs L}, \prettyref{eq: L vs SL for D}
and $\mathtt{h}_{\textrm{top}}\left(e^{kV}\right)=k\mathtt{h}_{\textrm{top}}\left(e^{V}\right)$
one now obtains
\begin{align*}
m\left(\frac{\ln\alpha}{\ln\left(\alpha\alpha_{\epsilon}\right)}\right)\ln L_{d_{k}}\left(e^{V}\right) & \leq\textrm{HD}\left(d_{k}\right)\ln L_{d_{k}}\left(e^{V}\right)\\
 & \leq\frac{\textrm{HD}\left(d_{k}\right)}{k}\ln L_{D}\left(e^{kV}\right)\\
 & \leq\frac{\textrm{HD}\left(d_{k}\right)}{k}\left[\ln16+\ln SL_{D}\left(e^{kV}\right)\right]\\
 & \leq\frac{\textrm{HD}\left(d_{k}\right)}{k}\ln16+\frac{1}{k}\mathtt{h}_{\textrm{top}}\left(e^{kV}\right)\\
 & =\frac{\textrm{HD}\left(d_{k}\right)}{k}\ln16+\mathtt{h}_{\textrm{top}}\left(e^{V}\right)\\
 & \leq\frac{m}{k}\ln16+\mathtt{h}_{\textrm{top}}\left(e^{V}\right).
\end{align*}
Letting $k\rightarrow\infty$, and noting that $\left(\frac{\ln\alpha}{\ln\left(\alpha\alpha_{\epsilon}\right)}\right)\rightarrow\frac{1}{2}$
as $\epsilon\rightarrow0$, one obtains the theorem.
\end{proof}
The above lemma now implies the main result of this subsection \prettyref{thm: lambda max vs htop thm.}.
\begin{proof}[Proof of \prettyref{thm: lambda max vs htop thm.}]
 As the previous \prettyref{lem:Entropy inequality} shows, for each
$\lambda>\frac{2}{m}\mathtt{h}_{\textrm{top}}$ we have 
\[
L_{d}\leq e^{\lambda},\quad\textrm{ for some }d\in\mathcal{D}_{\frac{1}{2}-}\left(Y\right).
\]
Using the semi-group property of the flow one obtains a positive constant
$c>0$ such that 
\[
d\left(e^{tV}x_{1},e^{tV}x_{2}\right)\leq ce^{\lambda t}d\left(x_{1},x_{2}\right).
\]
From $d^{g^{TY}}\lesssim d\lesssim\left(d^{g^{TY}}\right)^{\frac{1}{2}-\epsilon}$
this further gives
\begin{align}
d^{g^{TY}}\left(e^{tV}x_{1},e^{tV}x_{2}\right) & \leq ce^{2\lambda t}d^{g^{TY}}\left(x_{1},x_{2}\right)\nonumber \\
\left\Vert \left(e^{tV}\right)^{*}f\right\Vert _{C^{2}} & \leq ce^{2\lambda t}\left\Vert f\right\Vert _{C^{2}}\label{eq:growth of distance-1}
\end{align}
$\forall x_{1},x_{2}\in Y$, $f\in C^{2}\left(Y\right)$. The inequality
$\Lambda_{\textrm{max}}\leq2\lambda$ now follows easily from the
last equation and the definition of the maximal expansion rate.
\end{proof}

\subsection{\label{subsec:Compact-Lie-Groups} Compact Lie Groups}

Next we consider geodesic flows associated to bi-invariant metrics
on compact Lie groups. In this case the volume bound on the recurrence
set is the one given  below. It was essentially proved by Volovoy
in \cite[Prop. 4]{Volovoy-1990-I} and we refine the bound while following
his outline.
\begin{thm}
\label{thm:recurrence set bounds Lie group }Let $G$ be a compact
Lie group equipped with a bi-invariant metric $g$. The recurrence
sets for its geodesic flow satisfies the volume bounds
\begin{align}
\nu\left(S_{T,\varepsilon}^{1}\right) & =O\left(\varepsilon^{p-1}T^{p}\right)\label{eq:recurrence set bound Lie group}\\
\nu\left(S_{T,\varepsilon}^{1,e}\right) & =O\left(\varepsilon^{p-1}T^{p}\right),\label{eq:ext recurrence bound Lie group}
\end{align}
where $p=\textrm{rk}G$ is the rank of the Lie group.
\end{thm}

\begin{proof}
With $\mathfrak{g}$ being its Lie algebra, let $\exp:\mathfrak{g}\rightarrow G$
denote the exponential mapping of the Lie group. For bi-invariant
metrics, the path $A\exp tb$, $b\in\mathfrak{g}=T_{I}G$, gives the
geodesic through the point $A\in G$ in the direction $dL_{A}\left(b\right)\in T_{A}G$
(here $L_{A}$ denotes left multiplication by $A$).

From the compactness of the Lie group, one finds a uniform Lipschitz
constant $C>0$ such that 
\begin{equation}
d^{g}\left(A\exp tb,A\right)\leq Cd^{g}\left(\exp tb,I\right),\;\;\forall A\in G,\,b\in\mathfrak{g},\,t>0.\label{eq:Lipschitz constant on lie group}
\end{equation}
 This gives a constant $C_{1}$ such that 
\begin{align}
\nu\left(S_{T,\varepsilon}^{1}\right) & \leq C_{1}\nu_{I}\underbrace{\left\{ a\in S_{I}^{*}G|\exists t\in\left[\frac{1}{2}T_{0},T\right]\textrm{ s.t. }d^{g}\left(\exp ta,I\right)\leq\varepsilon\right\} }_{=S_{T,\varepsilon,I}^{1}}\label{eq:based recurrence Lie group}\\
\nu\left(S_{T,\varepsilon}^{1,e}\right) & \leq C_{1}\nu_{I}\underbrace{\left\{ a\in S_{I}^{*}G|d^{g}\left(a,S_{T,\varepsilon,I}^{1}\right)\leq\varepsilon\right\} }_{=S_{T,\varepsilon,I}^{1,e}}\label{eq:based extended recurrence Lie group}
\end{align}
for each $\varepsilon>0$ and $T>\frac{1}{2}T_{0}$. Here $\nu_{I}$
denote the induced measure on the unit sphere inside the dual Lie
algebra $\mathfrak{g}^{*}=S_{I}^{*}G.$ It thus suffices to estimate
the measure of the recurrence set based at the identity $S_{T,\varepsilon,I}^{1}$
on the right hand side above \prettyref{eq:based recurrence Lie group}.

Now let $H=\mathbb{T}^{p}\subset G$, $p=\textrm{rk}G$, be a maximal
torus. Similar recurrence sets $S_{T,\varepsilon,I}^{1,H}$, $S_{T,\varepsilon,I}^{1,e,H}$
as \prettyref{eq:based recurrence Lie group}, \prettyref{eq:based extended recurrence Lie group}
can be defined that is based at the identity $I\in H$ in the maximal
torus. In \cite[Prop. 3]{Volovoy-1990-I} the measure bound $\nu\left(S_{T,\varepsilon,I}^{1,H}\right)=O\left(\varepsilon^{p-1}T^{p}\right)$
for the based recurrence set inside an arbitrary torus was proved.
In fact, \cite[Cor. 3]{Volovoy-1990-I} showed that $S_{T,\varepsilon,I}^{1,H}\subset S_{I}^{*}H$,
and thus $S_{T,\varepsilon,I}^{1,e,H}$ too, could be covered with
a collection of radius $\varepsilon$- balls $\left\{ B_{\varepsilon}\left(h_{j}\right)|h_{j}\in S_{I}^{*}H\right\} _{j=1}^{M}$,
where $M=O\left(T^{p}\right)$. For a general group, any element $a\in S_{T,\varepsilon,I}^{1}\subset S_{I}^{*}G$
is conjugate to an element in the torus $\exp\left(t\textrm{ad}_{P}a\right)=P\left(\exp ta\right)P^{-1}\in H$
for some $P\in G$. It follows from \prettyref{eq:Lipschitz constant on lie group}
that the conjugates $\textrm{ad}_{P}a\in S_{T,C\varepsilon,I}^{1,H}$,
$\textrm{ad}_{P}b\in S_{T,C\varepsilon,I}^{1,e,H}$ are elements of
the based recurrence sets of the torus, for $a,b\in S_{T,\varepsilon,I}^{1},S_{T,\varepsilon,I}^{1,e}$
respectively, and hence in one of the $M=O\left(T^{p}\right)$ balls
$B_{\varepsilon}\left(h_{j}\right)$ of radius $\varepsilon$. It
thus suffices to prove the estimate $\nu\left(S_{j}\right)=O\left(\varepsilon^{p-1}\right)$
on the volumes of the conjugates $S_{j}\coloneqq G.B_{\varepsilon}\left(h_{j}\right).G^{-1}$,
$j=1,2,\ldots,M$ of these balls. This was also done by Volovoy in
\cite[pgs. 134-135]{Volovoy-1990-I}.
\end{proof}

\subsection{\label{subsec:Surface-of-revolution} Surface of revolution}

We now consider geodesic flows on compact surfaces of revolution. 

Namely, the manifold is now given as 
\begin{align}
X & =\left\{ \left(\rho\left(z\right)\cos\phi,\rho\left(z\right)\sin\phi,z\right)|\phi\in\left[0,2\pi\right],z\in\left[a_{-},a_{+}\right]\right\} \nonumber \\
 & \qquad\qquad\subset\mathbb{R}^{3}.\label{eq:surface of revolution}
\end{align}
Here $\rho:\left(a_{-},a_{+}\right)\rightarrow\left(0,\infty\right)$
is a smooth function satisfying $\lim_{z\rightarrow a_{\pm}}\rho\left(z\right)=0$,
$\lim_{z\rightarrow a_{\pm}}\rho'\left(z\right)=\mp\infty$. The surface
\prettyref{eq:surface of revolution} is thus obtained by rotating
the curve $\left(\rho\left(z\right),0,z\right)$, $a\leq z\leq b$,
around the $z$-axis (see \prettyref{fig:Surface-of-revolution}).
We shall further assume that the surface is strictly convex. That
is, the function $-\rho$ is strictly convex satisfying $-\rho''\left(z\right)>0$.
Thus $\rho$ is maximized at a unique $z_{0}\in\left[a_{-},a_{+}\right]$.
The curve $\gamma_{E}\coloneqq\left\{ \left(x,y,z\right)\in X|z=z_{0}\right\} $
shall be referred to as the \textit{equator}. The points $\left(0,0,a_{\pm}\right)$
are referred to as the north and south poles respectively.

\begin{figure}
\includegraphics[scale=0.6]{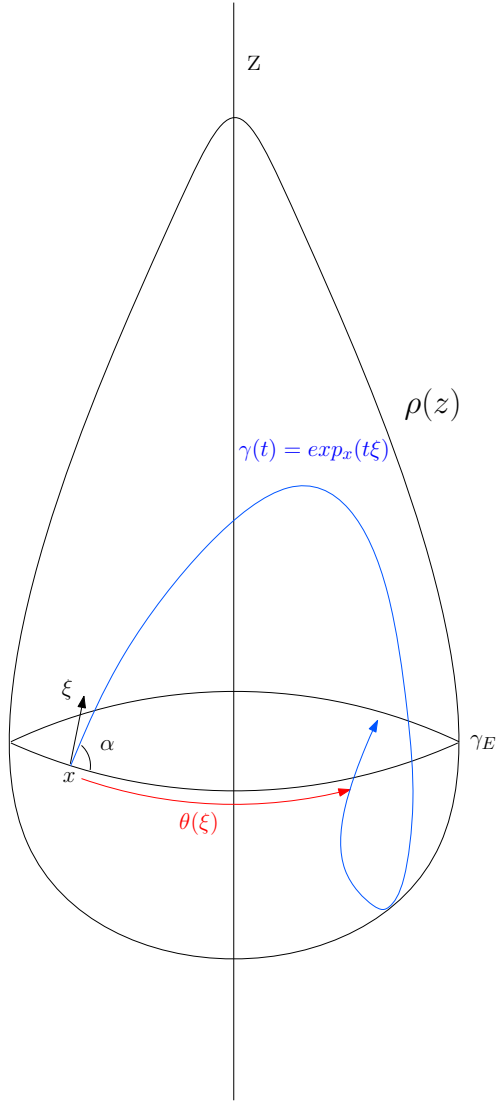}

\caption{\label{fig:Surface-of-revolution} Surface of revolution}
\end{figure}

The metric on $X$ is chosen to be the one induced from the Euclidean
embedding. The geodesic flow on surfaces of revolution is well understood.
Firstly, the Hamiltonian function can be computed
\[
H\left(\xi_{z},\xi_{\phi}\right)=\frac{1}{2}\left[\frac{\xi_{\phi}^{2}}{\rho^{2}}+\frac{\xi_{z}^{2}}{1+\rho_{z}^{2}}\right]
\]
in terms of the cylindrical coordinates $\left(z,\phi\right)$. The
angular momentum function $\xi_{\phi}$ Poisson commutes with the
above Hamiltonian and is hence preserved under the geodesic flow.
This conservation law can be rewritten in a more explicit form on
the base. The geodesic $\gamma=\left(z\left(t\right),\phi\left(t\right)\right)$
is the projection of the Hamilton trajectory 
\[
\left(z\left(t\right),\phi\left(t\right);\underbrace{\xi_{z}\left(t\right)}_{=\left(1+\rho_{z}^{2}\right)\dot{z}},\underbrace{\xi_{\phi}\left(t\right)}_{=\rho^{2}\dot{\phi}}\right),
\]
and thus the function $\rho^{2}\dot{\phi}=\xi_{\phi}\left(t\right)=c$
is constant along the flow. Denote by $\alpha$ the angle between
the velocity vector $\dot{\gamma}\left(t\right)$ at a point $\gamma\left(t\right)$
of the geodesic with the parallel through the point. It is easy to
compute $\rho\dot{\phi}=\left|\dot{\gamma}\right|\cos\alpha$. The
Hamiltonian/length $\frac{1}{2}\left|\dot{\gamma}\right|^{2}=H\left(\gamma\left(t\right)\right)$
is preserved along the flow. Thus the conservation law of $\xi_{\varphi}$
can thus be restated by saying that 
\begin{equation}
\rho\cos\alpha=c\;(\textrm{constant })\label{eq:Clairaut relation}
\end{equation}
along the flow. It is easy to integrate the above equation given the
initial point and velocity vector. Conversely, and from the uniqueness
existence of of geodesics, any unit speed (non-parallel) curve along
which the above relation \prettyref{eq:Clairaut relation} holds is
a geodesic. While the only parallel geodesic is the equator $\gamma_{E}$.
The last three lines constitute the statement of \textit{Clairaut's
theorem}. 
\begin{thm}
\label{thm:recurrence set bound surface revol.} Consider $\left(X^{2},g\right)$
a compact, strictly convex surface of revolution. Then volumes of
the recurrence sets of the geodesic flow \prettyref{eq:recurrence set}
satisfy the estimates
\begin{align}
\nu\left(S_{T,\varepsilon}^{1}\right) & =O\left(\varepsilon^{\frac{1}{r}}T^{1-\frac{1}{r}}\right)\label{eq:recurrence set surf. rev.}\\
\nu\left(S_{T,\varepsilon}^{1,e}\right) & =O\left(\varepsilon^{\frac{1}{r}}T^{1-\frac{1}{r}}\right),\label{eq:ex recurrence set surf rev.}
\end{align}
for $\varepsilon,T$ sufficiently small. Here $r$ \prettyref{eq:order of vanishing of E. return map}
is the maximum order of vanishing of the equatorial return map.
\end{thm}

\begin{proof}
It follows easily from the Clairaut relation \prettyref{eq:Clairaut relation},
that every geodesic necessarily intersects the equator $\gamma_{E}$
in a uniformly finite time. By compactness, it thus suffices to estimate
the measure of the based recurrence sets
\begin{align*}
S_{T,\varepsilon,x}^{1} & \coloneqq\left\{ \xi\in S_{x}^{*}X|\exists t\in\left[\frac{1}{2}T_{0},T\right]\textrm{ s.t. }d\left(\exp_{x}\left(t\xi\right),x\right)\leq\varepsilon,\;\exp_{x}\left(t\xi\right)\in\gamma_{E}\right\} \\
S_{T,\varepsilon,x}^{1,e} & \coloneqq\left\{ \xi\in S_{x}^{*}X|d\left(\xi,S_{T,\varepsilon,x}^{1}\right)\leq\varepsilon\right\} 
\end{align*}
for $x\in\gamma_{E}$.

Next for each $x\in\gamma_{E}$ on the equator, we define 
\begin{align}
\tau & :S_{x}^{*}X\rightarrow\mathbb{R}_{>0}\label{eq:first time of return}\\
\theta & :S_{x}^{*}X\rightarrow S^{1}=\left[0,2\pi\right]\label{eq:rotation angle}
\end{align}
as functions defined for elements in the cosphere $\xi\in S_{x}^{*}X$
above $x$. The first maps $\xi$ to the time of first return $\tau\left(\xi\right)$
of the geodesic $\gamma\left(t\right)=\exp_{x}\left(t\xi\right)$
to the equator. While the second maps $\xi$ to the angle of rotation
$\theta\left(\xi\right)$ of the equator needed to take $x$ to the
point of first return $\exp_{x}\left(\tau\left(\xi\right)\xi\right)$
(see \prettyref{fig:Surface-of-revolution}). There are standard formulas
for $\theta$ in terms of the defining function $\rho$ for the surface
(cf. \cite[Ch. 4B]{Besse78-book}). Now by Clairaut's equation \prettyref{eq:Clairaut relation}
the angle $\alpha$ between the geodesic and the equator will be the
same at $\exp_{x}\left(\tau\left(\xi\right)\xi\right)$. Thus by rotational
symmetry, we have the relations 
\[
\gamma\left(t+\tau\left(\xi\right)\right)=R_{\theta\left(\xi\right)}\gamma\left(t\right)
\]
 and the subsequent times of return will be $2\tau\left(\xi\right),3\tau\left(\xi\right),\ldots$
respectively. From the continuity of $\tau$ one can thus find a positive
$C>0$ such that for each $t>0$ there exists $\left|s\right|\leq C$
and $p\in\mathbb{Z}$, $\left|p\right|\leq Ct$, satisfying $\exp_{x}\left(t\xi\right)=R_{p\theta\left(\xi\right)}\gamma\left(s\right).$ 

One is thus reduced to estimating the measure of the set 
\begin{equation}
\tilde{S}_{T,\varepsilon,x}=\left\{ \xi\in S_{x}^{*}X|\exists p\in\mathbb{Z}\setminus\left\{ 0\right\} ,q\in\mathbb{Z},\:\left|p\right|,\left|q\right|\leq CT\textrm{ s.t. }\left|2\pi\theta\left(\xi\right)-\frac{q}{p}\right|\leq\frac{\varepsilon}{p}\right\} .\label{eq:return map recurrence set}
\end{equation}
From the definition of $r$ \prettyref{eq:order of vanishing of E. return map}
as the maximum vanishing order, one has $\left|\theta\left(\xi\right)-\theta\left(\xi_{0}\right)\right|\geq C\left|\xi-\xi_{0}\right|^{r}$
for each pair $\xi,\xi_{0}\in S_{x}^{*}X$ sufficiently close. The
volume of the set above is now easily estimated as 
\[
\nu_{x}\left(\tilde{S}_{T,\varepsilon,x}\right)=O\left(\sum_{p=1}^{CT}\left(\frac{\varepsilon}{p}\right)^{\frac{1}{r}}\right)=O\left(\varepsilon^{\frac{1}{r}}T^{1-\frac{1}{r}}\right)
\]
as required. 
\end{proof}

\section{\label{sec:Proofs-of-the-corollaries} Proofs of the Corollaries}

In this section we prove the three corollaries of our main \prettyref{thm:general Weyl law}.
They shall be based on the volume bounds on the recurrence sets from
\prettyref{sec:Examples-of-recurrence}.
\begin{proof}[Proofs of the Corollaries \prettyref{cor:Quant. Ber.}, \prettyref{cor:Lie group corollary},
\prettyref{cor: Surface of revolution corollary}]
 The pseudodifferential operator in all corollaries is the semiclassical
Laplacian $P_{h}=h^{2}\Delta_{g}$ and the interval to be $\left[a,1\right]$
with $a<0$. The principal symbol of the Laplacian is $p_{0}=\left|\xi\right|^{2}\in C^{\infty}\left(T^{*}X\right)$
is the norm square function on the cotangent bundle. While the sub-principal
symbol is zero $p_{1}=0$. Its relevant energy level is $\Sigma_{1}=S^{*}X$
the unit cosphere bundle of the manifold. This carries the contact
form $\alpha_{g}\in\Omega^{1}\left(S^{*}X\right)$ which is the restriction
$\alpha_{g}=\left.\alpha\right|_{S^{*}X}$ of the tautological one
form on the cotangent space. It is then well known that the Hamilton
vector field of the principal symbol $R_{g}=H_{\left|\xi\right|^{2}}$
is the Reeb vector field of this contact form. 

For Corollary \prettyref{cor:Quant. Ber.}, the manifold is taken
to be a negatively curved Riemannian manifold. In this case, the geodesic
flow is known to be an Anosov Reeb flow. Hence the recurrence set
volume bounds \prettyref{eq:recurrence measure Anosov}, \prettyref{eq:recurrence measure extended Anosov}
apply. We may now set $\varepsilon=h^{\delta}$, with $\delta=\frac{1}{4}$,
and $T=\frac{1}{4}T_{E}^{\ell,1}\left(h\right)=\frac{1}{4}\frac{\left|\ln h\right|}{\Lambda_{\textrm{max}}^{1}+\ell}$
in \prettyref{thm:general Weyl law}. Then \prettyref{eq:general remainder estimate}
becomes
\[
\left|R_{h}\right|\leq\textrm{vol}\left(S^{*}X\right)4\left(\Lambda_{\textrm{max}}^{1}+\ell\right)\left|\ln h\right|^{-1}+O\left(h^{\frac{1}{4}.\left(2n-1\right).\frac{\ell}{\Lambda_{\textrm{max}}^{1}+\ell}}+\left|\ln h\right|^{-2}+h^{\frac{1}{2}}\right)
\]
using the volume bounds \prettyref{eq:recurrence measure Anosov},
\prettyref{eq:recurrence measure extended Anosov}. Since $\ell>0$
is arbitrary, the equation \prettyref{eq:lambda max in remainder}
is proved. The second estimate \prettyref{eq:entropy in remainder}
is now a consequence of the inequality  $\Lambda_{\textrm{max}}^{1}\leq\frac{4}{n}h_{\textrm{top}}$
from \prettyref{thm: lambda max vs htop thm.}.

For Corollary \prettyref{cor:Lie group corollary}, the manifold $X=G$
is a compact Lie group equipped with a bi-invariant Riemannian metric.
We may now set $\varepsilon=h^{\delta}$, with $\left[0,\frac{1}{2}\right)\ni\delta=\begin{cases}
0, & p=1,\\
\frac{p+1}{3p+1}, & p>1,
\end{cases}$ and $T=h^{-\frac{p-1}{3p+1}}$ in \prettyref{thm:general Weyl law}.
Then \prettyref{eq:general remainder estimate} becomes 
\begin{align*}
R_{h} & =O\left(T^{-1}+\varepsilon^{p-1}T^{p}+h^{1-2\delta}\right)\\
 & =O\left(h^{\frac{p-1}{3p+1}}\right)
\end{align*}
 using the volume bounds \prettyref{eq:recurrence set bound Lie group},
\prettyref{eq:ext recurrence bound Lie group} as required.

Finally for the last Corollary \prettyref{cor: Surface of revolution corollary},
the manifold is a surface of revolution. Here we set $\varepsilon=h^{\delta},$
with $\delta=\frac{2r-1}{4r-1}\in\left[0,\frac{1}{2}\right)$ and
$T=h^{-\frac{1}{4r-1}}$ in \prettyref{thm:general Weyl law}. Then
\prettyref{eq:general remainder estimate} becomes 
\begin{align*}
R_{h} & =O\left(\varepsilon^{\frac{1}{r}}T^{1-\frac{1}{r}}+T^{-1}+h^{1-2\delta}\right)\\
 & =O\left(h^{\frac{1}{4r-1}}\right)
\end{align*}
using \prettyref{eq:recurrence set surf. rev.} as required. 

We remark that our choices of $\varepsilon$ and $T$ are optimal
based on the corresponding bounds for recurrence set volumes in each
case.
\end{proof}
\bibliographystyle{siam}
\bibliography{biblio}

\end{document}